\documentclass[11pt,a4paper]{amsart}

\usepackage{amssymb,mathrsfs,enumitem,amsmath,amsthm,bbold}
\usepackage[mathscr]{euscript}
\usepackage{color}
\usepackage{graphicx}
\usepackage{tipa}
\usepackage{tikz-cd}

\usepackage{import}
\usepackage{xifthen}
\usepackage{pdfpages}
\usepackage{transparent}

\usepackage[all]{xy}
\definecolor{Chocolat}{rgb}{0.36, 0.2, 0.09}
\definecolor{BleuTresFonce}{rgb}{0.215, 0.215, 0.36}
\definecolor{EgyptianBlue}{rgb}{0.06, 0.2, 0.65}

\newcommand\cyrillic[1]{
	{\fontencoding{OT2}\fontfamily{wncyr}\selectfont #1}
		}

\newcommand\mathcyr[1]{\text{\cyrillic{#1}}}
\newcommand\Sha{\textnormal{\mathcyr{Sh}}} 

\usepackage[colorlinks,final,hyperindex]{hyperref}
\hypersetup{citecolor=BleuTresFonce, urlcolor=EgyptianBlue, linkcolor=Chocolat}
\usepackage[OT2,OT1]{fontenc}

\usepackage{fourier}

\usepackage{tikz}
\usetikzlibrary{shapes,fit,positioning,calc,matrix}
\tikzset{
  optree/.style={scale=.5,thick,grow'=up,level distance=10mm,inner sep=1pt},
  comp/.style={draw=none,circle,fill,line width=0,inner sep=0pt},
  dot/.style={draw,circle,fill,inner sep=0pt,minimum width=3pt},
  circ/.style={draw,circle,inner sep=1pt,minimum width=4mm},
  emptycirc/.style={draw,circle,inner sep=1pt,minimum width=2mm},  
  root/.style={level distance=10mm,inner sep=1pt},
  leaf/.style={draw=none,circle,fill,line width=0,inner sep=0pt},
  nodot/.style={draw,circle,inner sep=1pt},
}

\newtheorem{theorem}{Theorem}[section]
\newtheorem{corollary}[theorem]{Corollary}

\newtheorem{lemma}[theorem]{Lemma}
\newtheorem{proposition}[theorem]{Proposition}

\theoremstyle{definition}
\newtheorem{example}[theorem]{Example}
\newtheorem{remark}[theorem]{Remark}
\newtheorem{definition}[theorem]{Definition}

\DeclareMathAlphabet{\pazocal}{OMS}{zplm}{m}{n}

\def\calA{\pazocal{A}}
\def\calB{\pazocal{B}}

\def\calE{\pazocal{E}}
\def\calF{\pazocal{F}}
\def\calG{\pazocal{G}}
\def\calH{\pazocal{H}}
\def\calI{\pazocal{I}}

\def\calM{\pazocal{M}}
\def\calN{\pazocal{N}}
\def\calO{\pazocal{O}}
\def\calP{\pazocal{P}}

\def\calR{\pazocal{R}}
\def\calS{\pazocal{S}}
\def\calT{\pazocal{T}}

\def\calW{\pazocal{W}}
\def\calX{\pazocal{X}}
\def\calY{\pazocal{Y}}

\DeclareMathOperator{\dgVect}{\mathsf{dgVect}}
\DeclareMathOperator{\sfC}{\mathsf{C}}

\DeclareMathOperator{\rmR}{\mathrm{R}}

\DeclareMathOperator{\Gr}{\mathsf{Gr}}
\DeclareMathOperator{\CGr}{\mathsf{CGr}}
\DeclareMathOperator{\GrCol}{\mathsf{GrCol}}

\DeclareMathOperator{\Tw}{Tw}

\DeclareMathOperator{\Aut}{Aut}
\DeclareMathOperator{\Hom}{Hom}

\DeclareMathOperator{\Des}{Des}

\DeclareMathOperator{\Rec}{\mathsf{Rec}}

\DeclareMathOperator{\grCom}{\mathsf{grCom}}

\DeclareMathOperator{\grGerst}{\mathsf{grGerst}}
\DeclareMathOperator{\grGrav}{\mathsf{grGrav}}
\DeclareMathOperator{\grHyperCom}{\mathsf{grHyperCom}}

\DeclareMathAlphabet{\mathbbold}{U}{bbold}{m}{n}

\def\k{\mathbbold{k}}

\newcommand{\ptwo}[2]{\ensuremath{ 
\vcenter{\hbox{\xymatrix@R=.4pc@C=.9pc{ 
    *+[o][F-]{#1}\ar@{-}[r]&*+[o][F-]{#2}
}}}}}

\begin{document}

\title{Reconnectads}

\author{Vladimir Dotsenko}

\address{ 
Institut de Recherche Math\'ematique Avanc\'ee, UMR 7501, Universit\'e de Strasbourg et CNRS, 7 rue Ren\'e-Descartes, 67000 Strasbourg CEDEX, France}

\email{vdotsenko@unistra.fr}

\author{Adam Keilthy}

\address{Department of Mathematical Sciences, Chalmers Technical University and the University of Gothenburg, SE-412 96 Gothenburg, Sweden}

\email{keilthy@chalmers.se}

\author{Denis Lyskov}

\address{National Research University Higher School of Economics, 20 Myasnitskaya street, Moscow 101000, Russia}

\email{ddl2001@yandex.ru}

\date{}

\begin{abstract}
We introduce a new operad-like structure that we call a reconnectad; the ``input'' of an element of a reconnectad is a finite simple graph, rather than a finite set, and ``compositions'' of elements are performed according to the notion of the reconnected complement of a subgraph. The prototypical example of a reconnectad is given by the collection of toric varieties of graph associahedra of Carr and Devadoss, with the structure operations given by inclusions of orbits closures. We develop the general theory of reconnectads, and use it to study the ``wonderful reconnectad'' assembled from homology groups of complex toric varieties of graph associahedra.
\end{abstract}

\maketitle

\setcounter{tocdepth}{1}
\tableofcontents

\section{Introduction}

In this paper, we define and study a new algebraic structure for which we propose the term \emph{reconnectad}\footnote{The intended pronunciation is \textipa{[rIk@nE"ktA:d]} with the stress on the last syllable, as if it were a French word.}.
It captures a certain self-similarity of stratifications of toric varieties whose dual polytopes are the so called graph associahedra, originally defined by Carr and Devadoss in \cite{MR2239078}. Their original goal was to find convex polytopes that would give tilings of Coxeter complexes of general Coxeter groups in the same way the associahedra of Stasheff give tilings of the Deligne--Mumford compactifications $\overline{\calM}_{0,n}(\mathbb{R})$ of the moduli space of real projective lines with marked points. The answer found in \cite{MR2239078}, originating in the theory of wonderful models for subspace arrangements due to De Concini and Procesi \cite{MR1366622}, turned out to be quite remarkable. Particular cases of that construction lead to well known families of polytopes: the associahedra \cite{MR0158400,MR146227}, the cyclohedra \cite{MR1295465}, and the permutahedra~\cite{MR305800}, and the general notion of a graph associahedron has been studied in a wide range of contexts, including algebraic combinatorics, complex algebraic geometry, polyhedral geometry, theoretical computer science, toric topology etc., see, for example, a highly non-exhaustive list of references \cite{MR4176852,MR3340192,MR2997505,MR3439304,MR2252112}. To mention two remarkable concrete examples of applications, face posets of certain graph associahedra are responsible for the ``correct'' combinatorics behind the algebraic structures arising in Floer homology \cite{MR2764887,MR3342676}, and the intersection theory on complex toric varieties of stellahedra is shown to be invaluable for studying combinatorical invariants of matroids, see\cite{MR4477425,Stellar}. \\

Specifically for the associahedra themselves, the corresponding complex toric varieties were studied in \cite{MR4072173} where the natural nonsymmetric operad structure on the homology of those varieties was studied; it was found that many properties of that nonsymmetric operad $\mathsf{ncHyperCom}$ is remarkably similar to the known properties of the symmetric operad $\mathsf{HyperCom}$ obtained as the homology of the operad of complex Deligne-Mumford compactifications $\overline{\calM}_{0,n}(\mathbb{C})$, controlling what is known as the tree level part of a cohomological field theory \cite{MR1702284}. One important deficiency, however, is that the operad $\mathsf{ncHyperCom}$ is not cyclic, meaning that there is no meaningful notion of a compatible scalar product on an algebra over that operad. To deal with this problem, the second author of this paper proposed to view the operadic structure on toric varieties of associahedra in a wider context. This paper is the first step of this bigger programme: we exhibit a way to organise toric varieties of all possible graph associahedra in a remarkable operad-like structure. Classically, components of operads are indexed by finite sets in a functorial way, so that automorphisms of finite sets acts on components. In the situation that we consider, components are indexed by simple connected graphs in a functorial way, and structure operations arise from inclusions of orbit closures of the tori; those orbit closures are products of toric varieties of smaller graph associahedra which is the self-similarity we referred to above. The combinatorics of orbit closures is closely related to the notion of the reconnected complement of a subgraph in a graph, hence the terminology that we propose. \\

The motivation of De Concini and Procesi was to construct a compactification of the complement of the given subspace arrangement by gluing in a normal crossing divisor. In our case, the hyperplane arrangement in question is the arrangement of coordinate hyperplanes, and hence we already are dealing with the complement of a divisor with normal crossings. However, something bizarrely remarkable happens: even in this extremely simple situation, there are other nontrivial choices of compactifications that turn out to be notable algebraic varieties. To give a simple example, if one chooses the \emph{maximal building set} and blows up all possible intersections, in other words, if one considers the case of the complete graph, the resulting varieties, first studied by Procesi \cite{MR1252661}, are found in the work Losev and Manin \cite{MR1786500}; they encode a certain version of the notion of a cohomological field theory \cite{LosevPolyubin,MR3112505}, and which were recently shown in \cite{DST} to play the role analogous to that played by the spaces $\overline{\calM}_{0,n+1}(\mathbb{C})$ in an analogue of the BV formalism arising in topological quantum mechanics~\cite{Lysov}. The operad-like structure on these varieties is very close to that of that describing permutads of Loday and Ronco \cite{MR2995045}, though without a total ordering of the underlying set.\\

A closely related though different operad-like structure in the context of graph associahedra was recently defined by Forcey and Ronco in \cite{https://doi.org/10.1112/jlms.12596} using the formalism of operadic categories of Batanin and Markl \cite{MR3406537}. Our approach is different in several ways. First, the formalism of \cite{https://doi.org/10.1112/jlms.12596} imposes a total ordering of the set of vertices and thus is closer to ``shuffle reconnectads''; second, our approach allows us to view reconnectads as monoids in a certain monoidal category, similarly to how it can be done for operads. The advantage of such approach is that there is a very powerful existing range of ideas and methods available, and we are able to use those ideas and methods in a very meaningful way. If one wishes to place our approach under the umbrella of a general formalism for studying generalisations of operads, Feynman categories of Kaufmann and Ward \cite{MR3636409} furnish an example of such a formalism; however, for us, several other more concrete approaches turn out to be available. \\

It is worth mentioning that our work may be viewed as a shadow of a much more general theory developed by Coron \cite{CoronLattices} for the full generality of geometric lattices and their building sets. However, since unlike the \emph{op. cit.}, our focus is on a very particular case (of Boolean lattices and their graphical building sets), a lot of general statements become much more concrete, some generally very complicated objects become much simpler, and, as a result, some elegant combinatorial patterns shine through. We hope that our work will help in further understanding of the Feynman category of built lattices of \cite{CoronLattices}, and in placing constructions from toric geometry like the Bergman fan of a matroid \cite{MR2191630} in the categorical context.\\

An intriguing question that is for the moment left outside the scope of this paper is to give a generalisation of the Batalin-Vilkovisky operad $\mathsf{BV}$ in the context of reconnectads, and to compute its homotopy quotient by the circle action, at least on the algebraic level, generalising some of the results of \cite{DST}. This task that is not obvious for a number of reasons. First, it is absolutely crucial for reconnectads to have no nontrivial elements associated to the empty graph, which is where the BV operator would be expected to appear. Second, for the triangle graph there are two contradicting wishes of what the corresponding component of the BV reconnectad should be: viewing the triangle as a cycle relates it to path graphs and the $\mathrm{ncBV}$ operad of \cite{MR4072173}, while viewing it as a complete graph relates it to the $\mathrm{tBV}$ twisted associative algebra of \cite{DST}, and reconciling those two relationships is not an easy task. We hope to address this question elsewhere.\\

The paper is organised as follows. In Section \ref{sec:recollections}, we recall some necessary background information. In Section \ref{sec:graphical-vars}, we present three equivalent constructions of toric varieties of graph associahedra, including an interpretation as a ``graphical grassmannian'' which allows us to give a new explicit description of the stratification by toric orbits (Theorem \ref{th:orbits}). In Section \ref{sec:reconnectads}, we give several equivalent definitions of a reconnectad, identify two known particular cases, and define the ``commutative reconnectad''. In Section \ref{sec:algebra-reconnectads}, we develop a wide range of methods for studying algebraic reconnectads (reconnectads whose components are vector spaces). Finally, in Section \ref{sec:wonderful}, we define the gravity reconnectad and determine its presentation by generators and relations (Theorem \ref{th:gravpresent}), obtain a presentation by generators and relations of the ``wonderful reconnectad'' formed by the collection of homologies of all complex toric varieties of graph associahedra (Theorem \ref{th:homologyhyper}), and give an algebraic and a geometric proof of Koszul duality between these reconnectads and of the Koszul property of both of them. Geometrically, the reconnectadic Koszul duality is implemented by the compactifications to open orbits of the torus action on these varieties (Proposition \ref{prop:delignekoszul}); this is analogous to the celebrated result of Getzler \cite{MR1363058}.\\

\subsection*{Acknowledgements} Thanks are due to Anton Khoroshkin and Sergey Shadrin for useful discussions at various stages of preparation of this paper, and to Guillaume Laplante-Anfossi for comments on its first draft. The first author is grateful to Basile Coron for conversations about the general formalism of \cite{CoronLattices} and to Cl\'ement Dupont for discussions of wonderful compactifications and the Deligne spectral sequence.

The first author was supported by Institut Universitaire de France, by Fellowship USIAS-2021-061 of the University of Strasbourg Institute for Advanced Study through the French national program ``Investment for the future'' (IdEx-Unistra), and by the French national research agency (project ANR-20-CE40-0016). The second author was supported by the ``Postdoctoral program in Mathematics for researchers from outside Sweden'' (project KAW 2020.0254). Results of Section 6 were obtained under support of the grant RSF 22-21-00912 of Russian Science Foundation. Some parts of this work were completed during the time the first two authors spent at Max Planck Institute for Mathematics in Bonn, and they wish to express their gratitude to that institution for hospitality and excellent working conditions.

\section{Background, notations, and recollections}\label{sec:recollections}

By $\sfC$ we denote a category that has all coproducts (which we denote $\oplus$) and an initial object (which we denote by $0$), and is equipped with a symmetric monoidal structure (for which we denote the monoidal product by $\otimes$ and the unit object by $1_{\sfC}$) that distributes over coproducts. In most applications, $\sfC$ will be either the category of ``spaces'' (topological spaces, projective varieties etc.), or the category of ``modules'' (vector spaces, chain complexes etc.). For a finite set $I$ and a family of objects $\{V_i\}_{i\in I}$ of $\sfC$, the unordered monoidal product of these objects is defined by 
 \[
\bigotimes_{i\in I} V_i:=\left(\bigoplus_{(i_1,\ldots,i_n) \text{ a total order on } I} V_{i_1}\otimes\cdots\otimes V_{i_n}\right)_{\Aut(I)}.
 \]
In cases where we work with ``modules'', we assume that all of them are defined over a field $\k$ of zero characteristic. All chain complexes are graded homologically, with the differential of degree~$-1$. To handle suspensions of chain complexes, we introduce a formal symbol~$s$ of degree~$1$, and define, for a graded vector space~$V$, its suspension $sV$ as $\k s\otimes V$.

\subsection{Toric varieties}

Let us give a short summary of basics of toric varieties, referring the reader to \cite{MR495499,MR1234037} for more details.

We denote by $\mathbb{G}_m$ the algebraic group $\mathop{\mathrm{Spec}}\left(\k[x,x^{-1}]\right)$, i.e. the multiplicative group $\k^\times$.
An \emph{algebraic torus} is a product of several copies of $\mathbb{G}_m$. A \emph{toric variety} is a normal algebraic variety $X$ that contains a dense open subset $U$ isomorphic to an algebraic torus, for which the natural torus action on $U$ extends to an action on $X$. 

Toric varieties may be constructed from the combinatorial data of a \emph{lattice} (free finitely generated Abelian group) $N$ and a \emph{fan} (collection of strongly convex rational polyhedral cones closed under taking intersections and faces) $\Sigma$ in $N\otimes_{\mathbb{Z}}\mathbb{R}$. Each cone in a fan gives rise to an affine variety, the affine spectrum of the semigroup algebra of the dual cone. Gluing these affine varieties together according to face maps of cones gives an algebraic variety denoted $X(\Sigma)$ and called the toric variety associated to the fan $\Sigma$. 
    
It is known that a toric variety $X(\Sigma)$ is projective if and only if $\Sigma$ is a normal fan of a convex polytope $\calP$, uniquely determined from $\Sigma$ up to normal equivalence. In such situation, we also use the notation $X(\calP)$ instead of $X(\Sigma)$. The variety $X(\calP)$ is smooth if and only if $\calP$ is a \emph{Delzant polytope}, that is a polytope for which the slopes of the edges adjacent to each given vertex form a basis of the lattice~$N$. 

For a complex toric variety $X_{\mathbb{C}}(\calP)$ corresponding to an $n$-dimensional Delzant polytope $\calP$, the Betti numbers of $X_{\mathbb{C}}(\calP)$ are given by the coefficients $h_i$ of the $h$-polynomial of $\calP$ 
 \[
\sum_{i=0}^n h_it^i=\sum_{i=0}^n f_i(t-1)^i,   
 \]
where $f_i$ denotes the number of faces of $\calP$ of dimension $i$.

\subsection{Wonderful compactifications of subspace arrangements}\label{sec:rec-wond}

We also give a short summary of basics of (projective) wonderful compactifications of subspace arrangements in the particular case of the the coordinate subspace arrangements, referring the reader to \cite{MR1366622,MR2038195,MR2746338} for details as well as for the general theory.

Let $I$ be a finite set. We shall consider the vector space $\k^I$, and the collection of coordinate subspaces $C_T:=\{x_t=0\colon t\in T\}$ in that vector space. The Boolean lattice $2^I$ of all subsets of $I$ ordered by inclusion can be identified with the poset of all sums of the subspaces $C_T$, ordered by reverse inclusion. 

A \emph{building set} of $2^I$ is a collection $G$ of subsets of $2^I\setminus\{\emptyset\}$ such that for each $x\in 2^I$ the natural map
 \[
\prod_{g\in \max(G\cap [\emptyset,x])}[\emptyset,g]\to 
[\emptyset,x]
 \]
sending a tuple of elements to their union is an isomorphism of posets. This definition is valid in full generality of atomic lattices \cite{MR2038195}, and in the case of $2^I$ is equivalent to containing all singletons $\{i\}$ and containing, together with any two elements $x,y\in 2^I$ with $x\cap y\ne\emptyset$, their union $x\cup y$. We shall only consider building sets that contain the whole set $I$.

Since every building set contains all singletons, for each building set $G$, the complement in $\k^I$ of the union of the subspaces $C_T$ for $T\in G$ is the algebraic torus $\mathbb{G}_m^I$. Since $\mathbb{G}_m^I=(\k^\times)^I$, we have the maps 
 \[
\mathbb{G}_m^I/\mathbb{G}_m\to\mathbb{P}(\k^I/(\k^T)^\bot)
 \]
for all $T\in G$, and therefore a map
\begin{equation}\label{eq:torusmap}
\mathbb{G}_m^I/\mathbb{G}_m\to\mathbb{P}(\k^I)\times\prod_{T\in G} \mathbb{P}(C_T) .
\end{equation} 
The \emph{projective wonderful compactification} $\overline{Y}_G$ of $\mathbb{G}_m^I/\mathbb{G}_m$ is the closure of the image of $\mathbb{G}_m^I/\mathbb{G}_m$ under the map \eqref{eq:torusmap}. It is known that for every building set $G$, $\overline{Y}_G$ is a smooth projective irreducible variety. The natural projection map $\pi\colon \overline{Y}_G\to \mathbb{P}(\k^I)$ is surjective, and restricts to an isomorphism on $\mathbb{G}_m^I/\mathbb{G}_m$. Additionally, the complement $\overline{Y}_G\setminus (\mathbb{G}_m^I/\mathbb{G}_m)$ is a divisor with normal crossings. Its irreducible components $D_T$ are in one-to-one correspondence with elements $T\in G\setminus\{I\}$, and we have 
    \[
\pi^{-1}(\mathbb{P}(C_T))=\bigcup_{T\leq T'} D_{T'} \ .  
    \]
Intersections of the divisors $D_T$ are described using the combinatorial notion of a nested set. There are several closely related notion of a nested set, and it will be important for us to clearly distinguish between them. A set $\tau\subset 2^{G}$ is said to be \emph{nested} if for any elements $T_1,T_2,\ldots,T_k\in \tau$ which are pairwise incomparable in $2^I$, we have $T_1\cup T_2\cup \cdots\cup T_k\notin \tau$. The collection of all nested sets forms a simplicial complex $\tilde{N}(I,G)$ with the set of vertices $G$. Topologically, $\tilde{N}(I,G)$ is a cone with apex $\{I\}$, and the link of $\{I\}$ is denoted $N(I,G)$ and called the \emph{nested set complex} with respect to $G$. For a subset $\tau\subset 2^G$, the intersection $D_\tau=\bigcap_{T\in \tau} D_T$ is non-empty if and only if $\tau\in N(I,G)$. We shall also use the \emph{augmented nested set complex} obtained by adding to $N(I,G)$ one $-1$-simplex $\varnothing$. For our purposes, it will be convenient to realize the augmented nested set complex as the set $N^+(I,G)$ of all nested sets containing $I$ as an element; removing $I$ from such a nested set gives a bijection with the augmented nested set complex as described above. 

\begin{example}
Let $G$ be the maximal building set of $2^{\{1,2\}}$, that is, $G=\{\{1\},\{2\},\{1,2\}\}$. Then 
 \[
N(\{1,2\},G)=\{
\{\{1\}\},
\{\{2\}\},
\{\{1\},\{2\}\},
 \]
while 
 \[ 
N^+(\{1,2\},G)=\{
\{\{1,2\}\},
\{\{1\},\{1,2\}\},
\{\{2\},\{1,2\}\},
\{\{1\},\{2\},\{1,2\}\}
\}
 \]
\end{example}

Let us recall a useful way to visualise elements of $N^+(I,G)$ using labelled rooted trees going back to \cite{MR2191630,MR2487491}. Suppose that $\tau\in N^+(I,G)$. We associate to it a rooted tree $\mathbb{T}_\tau$ whose vertices are labelled by disjoint subsets of $V_\Gamma$ in the following way. The set of vertices of $\mathbb{T}_\tau$ is $\tau$, and a vertex $T$ is a parent of another vertex $T'$ if in the restriction of the order of $2^I$ to $\tau$ the element $T$ covers the element $T'$, and each vertex $T$ is labelled by the subset 
 \[
\lambda(T):=T\setminus \bigcup_{T'\in \tau, T'\subset T} T' .
 \]
\begin{example}
Let $G$ be the maximal building set of $2^{\{1,2,3,4\}}$. In Figure~\ref{fig:nested-tree}, we give an example of the labelled rooted tree $\mathbb{T}_\tau$ corresponding to the element $\tau=\{\{1\},\{3,4\},\{1,2,3,4\}\}$ of $N^+(I,G)$. (We always depict rooted trees in the way that the root is at the bottom.) 
\begin{figure}
\caption{Rooted tree associated to a nested set}
\label{fig:nested-tree}
 \[
\tau=\{\{1\},\{3,4\},\{1,2,3,4\}\}
\qquad
\Leftrightarrow
\qquad
\mathbb{T}_\tau=\vcenter{\hbox{\begin{tikzpicture}[scale=0.7]
     \draw[thick] (0,0.5)--(-1,2);
     \draw[thick] (0,0.5)--(1,2);
     
     \draw[fill=white, thick] (0,0.5) circle [radius=15pt];
     \draw[fill=white, thick] (-1,2) circle [radius=15pt];
     \draw[fill=white, thick] (1,2) circle [radius=20pt];
    \node at (0,0.5) {\scalebox{1}{$\{2\}$}};
    \node at (-1,2) {\scalebox{1}{$\{1\}$}};
    \node at (1,2) {\scalebox{1}{$\{3,4\}$}};
    
    \node (n) at (0,-0.5) {};
    \end{tikzpicture}}}
 \]
 \end{figure}
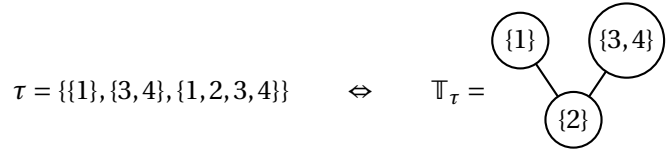  
\end{example}

\subsection{Graphs}

By a \emph{graph} we shall mean what is usually referred to as a finite simple graph. In other words, the datum of a graph is a pair $\Gamma=(V_\Gamma,E_\Gamma)$, where $V_\Gamma$ is a (possibly empty) finite set of \emph{vertices} and $E_\Gamma\subset V_\Gamma^2$ is a symmetric irreflexive relation on $V_\Gamma$ (two vertices $v_1,v_2$ such that $(v_1,v_2)\in E_\Gamma$ are said to be \emph{connected by an edge}). We shall frequently use the following particular examples of graphs:
\begin{itemize}
\item the empty graph $(\emptyset,\emptyset)$, which we denote simply by $\emptyset$,
\item the complete graph $K_n$ on the vertex set $\{1,\ldots,n\}$, whose edges are all possible pairs $(i,j)$ with $i\ne j$,
\item the stellar graph $S_n$ on the vertex set $\{0,\ldots,n\}$, whose edges are all possible pairs $(0,i)$ with $1\le i\le n$,
\item the path graph $P_n$ on the vertex set $\{1,\ldots,n\}$, whose edges are all possible pairs $(i,i+1)$ with $1\le i\le n-1$, 
\item the cycle graph $C_n$ on the vertex set $\mathbb{Z}/n\mathbb{Z}$, whose edges are all possible pairs $(i,i+1)$ with $i\in\mathbb{Z}/n\mathbb{Z}$. 
\end{itemize}

The \emph{disjoint union of two graphs} is defined by the formula 
 \[
\Gamma\sqcup \Gamma' = (V_\Gamma\sqcup V_{\Gamma'}, E_\Gamma\sqcup E_{\Gamma'}). 
 \]
We shall denote by $C_\Gamma\subset V_\Gamma\times V_\Gamma$ the minimal equivalence relation containing~$E_\Gamma$. If every two vertices of $\Gamma$ belong to the same equivalence class of $C_\Gamma$, we say that our graph $\Gamma$ is \emph{connected}, otherwise, $\Gamma$ coincides with the disjoint union of its \emph{connected components} $\mathrm{Conn}(\Gamma)$:
 \[
\Gamma=\bigsqcup_{\Gamma'\in\mathrm{Conn}(\Gamma)}\Gamma'. 
 \]
Recall that for a subset $V$ of $V_\Gamma$, the corresponding \emph{induced subgraph} is the graph $\Gamma_V$ whose vertex set is $V$ and whose edges are precisely the edges that exist in $\Gamma$. 

\section{Graphical varieties}\label{sec:graphical-vars}

In this section, we shall discuss a remarkable collection of algebraic varieties associated to graphs. It relies on a beautiful combinatorial construction that emerged independently in \cite{MR2239078,MR2487491,MR2470574,MR2252112}. 

\subsection{Graph associahedra, graphical building sets, graphical nested sets}

Let $\Gamma$ be a graph. We define \emph{graphical building set} $G_\Gamma$ of the Boolean lattice $2^{V_\Gamma}$ as the set of all subsets $V\subseteq V_\Gamma$ for which the induced graph $\Gamma_V$ is connected (in the context of graph associahedra such subsets are often referred to as the tubes of~$\Gamma$, following \cite{MR2239078}). 

The \emph{graph associahedron} $\calP\Gamma$ is the convex polyhedron obtained as follows. Take a simplex with the vertex set $V_\Gamma$, and index each of its faces by the vertices it misses. Then truncate the faces corresponding to elements of $G_\Gamma$, in the increasing order of dimension. It follows from \cite[Th.~2.6]{MR2239078} that the combinatorial type of $\calP\Gamma$ does not depend on the way in which the partial order by dimension is refined to a total order. As an example, suppose that $\Gamma$ is the path graph $P_3$. Then the vertices of the simplex are the two-dimensional subsets $\{1,2\}$, $\{1,3\}$, and $\{2,3\}$, of which $\{1,3\}\notin G_\Gamma$. Thus, we should truncate two vertices of the triangle, obtaining the pentagon, which is the classical Stasheff associahedron $K^2$.

For the two connected graphs on three vertices, the corresponding polyhedra (in this case, polygons) are displayed in Figure~\ref{fig:hedra}. 

\begin{figure}[h]
\begin{center}
  \includegraphics[scale=0.25]{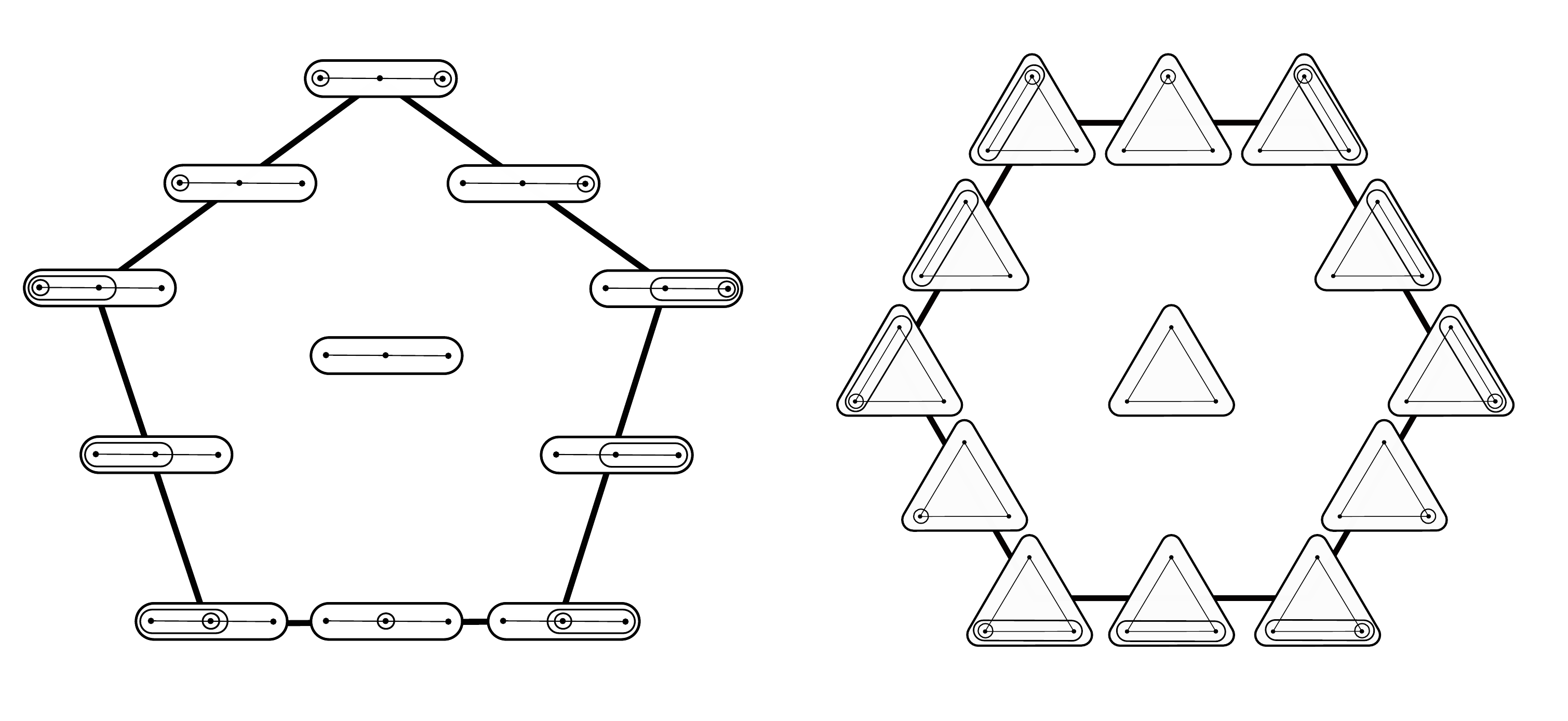}
\end{center}
  \caption{Graph associahedra for the path $P_3$ and the cycle $C_3$}
  \label{fig:hedra}
\end{figure}

In the context of graph associahedra, one often uses the combinatorics of ``tubings''; however, the precise definition of that notion varies throughout the literature, for instance the tubings of \cite{MR2239078} are elements of the nested set complex $N(V_\Gamma,G_\Gamma)$, while tubings of \cite{https://doi.org/10.1112/jlms.12596} are elements of the realisation $N^+(V_\Gamma,G_\Gamma)$ of the augmented nested set complex. Both notions arise naturally in different aspects of the story; to simplify the notation, we denote  
$N(V_\Gamma,G_\Gamma)$ by $N(\Gamma)$ and $N^+(V_\Gamma,G_\Gamma)$ by $N^+(\Gamma)$. It follows easily from the definition that both $N(\Gamma)$ and $N^+(\Gamma)$ consist of sets $\{T_1,\ldots,T_r\}$ of  subsets of $V_\Gamma$ for which each graph $\Gamma_{T_i}$ is connected and for all $i\ne j$ either one of the subsets $T_i$ and $T_j$ is a subset of the other or $\Gamma_{T_i\cup T_j}=\Gamma_{T_i}\sqcup \Gamma_{T_j}$; the only difference is that in the case of $N(\Gamma)$ we require that all $T_i$ are proper subsets of $V_\Gamma$, and in the case of $N^+(\Gamma)$ one of them must be equal to $V_\Gamma$. Two examples of nested sets in $N^+(\Gamma)$ for two different graphs $\Gamma$ are displayed in Figure~\ref{fig:nested}. 

\begin{figure}[h]
  \centering
  \includegraphics[width=0.7\linewidth]{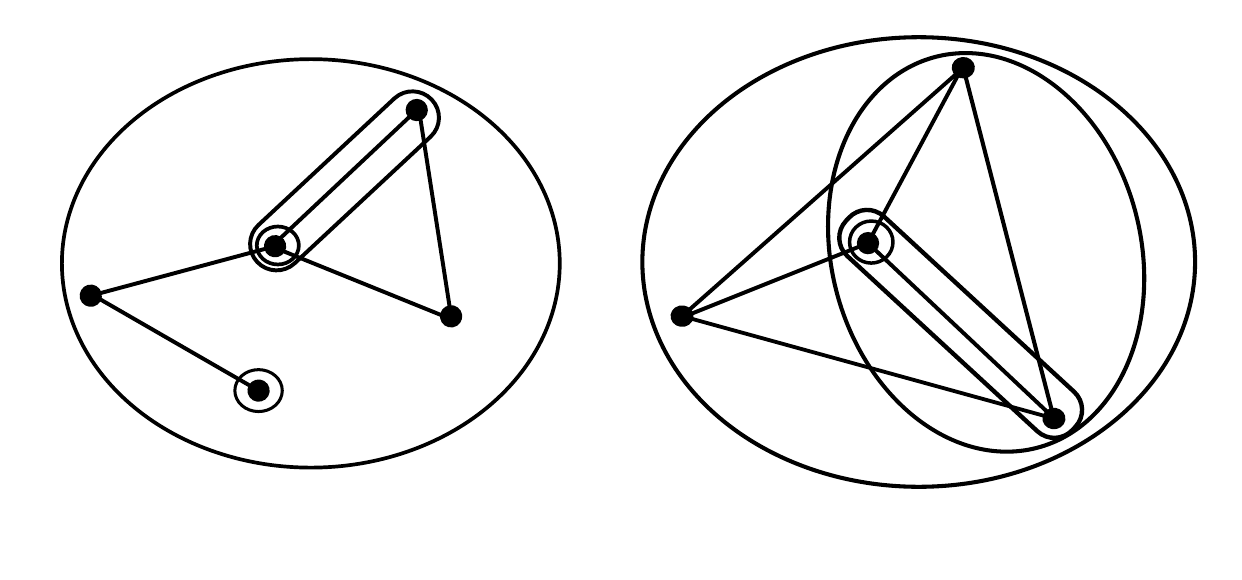}
  \caption{Examples of nested sets in $N^+(\Gamma)$}
  \label{fig:nested}
\end{figure}

It is also prudent to note that in \cite{MR4425832} the notion of a ``nesting'' of a graph are used, where edges rather than vertices are in the spotlight; in our terminology, those correspond to the nested sets of the derived graph \cite{MR262097}.

\subsection{Toric varieties of graph associahedra and wonderful compactifications}

In the context of toric geometry, each truncation of a face of a polytope $\calP$ corresponds to a certain blow up of the toric variety $X(\calP)$. In our case, the simplex with the vertex set $V_\Gamma$ corresponds to the projective space $\mathbb{P}(\k^{V_\Gamma})$, and the toric variety $X(\calP\Gamma)$ can be obtained from the latter by iterated blow ups  centered at coordinate subspaces $\mathbb{P}(U_T)$ corresponding to elements of $G_\Gamma$, in the increasing order of dimension. 

\begin{proposition}
For the graphical building set $G_\Gamma$ of the Boolean lattice $2^{V_\Gamma}$, the corresponding projective wonderful compactification  $\overline{Y}_{G_\Gamma}$ of $\mathbb{G}_m^{V_\Gamma}/\mathbb{G}_m$ is naturally isomorphic to the toric variety $X(\calP\Gamma)$.
\end{proposition}

\begin{proof}
Recall that $\overline{Y}_{G_\Gamma}$ is the closure of the image of the map
 \[
\mathbb{G}_m^{V_\Gamma}/\mathbb{G}_m\to\mathbb{P}(\k^{V_\Gamma})\times\prod_{T\in G_\Gamma} \mathbb{P}(C_T).
 \]
The codomain of this map has the obvious action of $\mathbb{G}_m^{V_\Gamma}/\mathbb{G}_m$, and the map is equivariant, so the algebraic torus acts on $\overline{Y}_{G_\Gamma}$ with an open orbit, therefore the variety $\overline{Y}_{G_\Gamma}$ is toric. Moreover, the projection from $\overline{Y}_{G_\Gamma}$ on each individual factor $\mathbb{P}(\k^{V_\Gamma}/g^\bot)$ is surjective, and hence by \cite[Prop.~8.1.4]{MR2394437}, the polytope corresponding to the toric variety $\overline{Y}_{G_\Gamma}$ is the Minkowski sum of simplices corresponding to elements of $G_\Gamma$. The same is known \cite{MR2487491} for each graph associahedron $\calP\Gamma$. 
\end{proof}

\subsection{Graphical grassmannians}

Let us give an equivalent description of the toric variety of $X(\calP\Gamma)$ as the parameter space of certain collections of subspaces in a vector space; for that reason, we shall call that parameter space a ``graphical grassmannian''. This notion is very close to that of a type A brick manifold \cite{MR4072173,MR3512647} in the case where $\Gamma$ is a path graph.

\begin{definition}
Let $\Gamma$ be a graph. The \emph{graphical grassmannian} $X(\Gamma)$ parametrises collections $\{H_T\}_{T\in G_\Gamma}$ of subspaces of $\k^{V_\Gamma}$ satisfying the following properties:
\begin{itemize}
\item $H_T\subset \k^{T}\subset \k^{V_\Gamma}$,
\item $\dim H_T=|T|-1$,
\item if $T\subset T'$, then $H_T\subset H_{T'}$.
\end{itemize}
\end{definition}

\begin{proposition}
The graphical grassmannian $X(\Gamma)$ is naturally isomorphic to the projective wonderful compactification $\overline{Y}_{G_\Gamma}$ of $\mathbb{G}_m^{V_\Gamma}$. 
\end{proposition}

\begin{proof}
We shall see that this is a simple consequence of projective duality, exactly as in \cite[Th.~5.2.1]{MR4072173}. The variety $\overline{Y}_{G_\Gamma}$ is the closure of the image of the map
 \[
\mathbb{G}_m^{V_\Gamma}/\mathbb{G}_m\to\mathbb{P}(\k^{V_\Gamma})\times\prod_{T\in G_\Gamma} \mathbb{P}(C_T).
 \]
Using the canonical basis of $\k^{V_\Gamma}$, we may identify the linear dual $(\k^{V_\Gamma})^*$ with $\k^{V_\Gamma}$. Under this identification, a point of $\mathbb{P}(C_T)$ is identified with a hyperplane in $\k^{T}$, so a point in $\prod_{T\in G_\Gamma} \mathbb{P}(C_T)$ gives rise to a collection of subspaces $\{H_T\}_{T\in G_\Gamma}$ satisfying the first two conditions. Moreover, the third condition is also satisfied because of the compatibility of the canonical basis of $\k^{V_\Gamma}$ with the canonical bases of all $\k^T$. The procedure we described is clearly invertible: to each point of the graphical grassmannian, we may associate a point in $\prod_{T\in G_\Gamma} \mathbb{P}(C_T)$. Moreover, if point of $\k^{V_\Gamma}$ associated to $H_{V_\Gamma}$ is in the complement of the coordinate hyperplanes, all other subspaces $H_V$ are reconstructed uniquely as $H_{V_\Gamma}\cap\k^{V}$, so the open piece of the graphical grassmannian maps isomorphically to the open piece of $\overline{Y}_{G_\Gamma}$, and hence the image of the inverse map is precisely $\overline{Y}_{G_\Gamma}$.
\end{proof}

The torus action is completely transparent in this equivalent description of our varieties. Indeed, the torus $(\mathbb{G}_m)^{V_\Gamma}$ acts on $\k^{V_\Gamma}$ by diagonal matrices, and this action leads to the action on the collections $\{H_T\}_{T\in G_\Gamma}$ in the obvious way. Clearly, the diagonally embedded torus $\mathbb{G}_m\subset (\mathbb{G}_m)^{V_\Gamma}$ acts trivially, so we have the action of the quotient torus $(\mathbb{G}_m)^{V_\Gamma}/\mathbb{G}_m$. 

\subsection{Torus orbits} \label{sec:toric-orbits}
From the point of view of toric geometry, the varieties $X(\calP\Gamma)$ have natural stratifications where the open boundary strata are toric orbits with stabilizers given by the tori corresponding to faces of $\calP\Gamma$. From the point of view of wonderful compactifications, the varieties $\overline{Y}_{G_\Gamma}$ have natural stratifications where the closed boundary strata are intersections $D_\tau$ of the divisors $D_T$ for all nested sets $\tau$. Our previous comparisons of these two approaches show that the two stratifications are equivalent. We shall now give a direct description of the open strata in the language of graphical grassmannians, generalising the stratification of brick manifolds given in \cite[Def.5.1.4]{MR4072173}; since we do not restrict ourselves to the boundary, it is more natural to use the augmented nested set complex $N^+(\Gamma)$ to index the strata. 

Suppose that $\tau\in N^+(\Gamma)$. We define a variety $X(\Gamma,\tau)$ as the set of collections $\{H_T\}_{T\in G_\Gamma}\in X(\Gamma)$ satisfying two conditions, both expressed in terms of the labels $\lambda(T)$ of the tree $\mathbb{T}_\tau$ described in Section \ref{sec:rec-wond}:  
\begin{itemize}
\item For each $T\in\tau$ with the parent $T'$ in $\mathbb{T}_\tau$ and for each $v\in \lambda(T')$ such that $T\cup\{v\}\in G_\Gamma$, we have 
 \[
H_{T\cup\{v\}}=\k^{T}.
 \] 
\item For each $T\in\tau$ and each $v_1\ne v_2\in \lambda(T)$, 
if for some $T'\subset T\setminus\{v_1,v_2\}$ we have 
 \[
T'\in \tau\cup\{\emptyset\}\text{  and } T'\cup\{v_1\}\cup\{v_2\}\in G_\Gamma,
 \] 
then the subspace $H_{T\cup\{v_1\}\cup\{v_2\}}$ is different from $\k^{T\cup\{v_1\}}$ and from $\k^{T\cup\{v_2\}}$.
\end{itemize}
Note that the first condition is a ``boundary'' condition and that the second one is an ``open'' condition.

As a sanity check, let us consider $\tau=\{V_\Gamma\}$. In this case, the first condition is empty, and the second condition, once restricted to edges, is easily seen to describe precisely the open piece discussed above ($H_{V_\Gamma}$ is in the complement of all coordinate hyperplanes), as expected.

\begin{theorem}\label{th:orbits}
The subsets $X(\Gamma,\tau)$ for $\tau\in N^+(\Gamma)$ describe the stratification of $X(\Gamma)$ by torus orbits.
\end{theorem}

\begin{proof}
Let us first show that 
 \[
X(\Gamma)=\bigcup_{\tau\in N^+(\Gamma)} X(\Gamma,\tau).
 \]
Given a collection of subspaces parametrised by $X(\Gamma)$, let consider all $T\in G_\Gamma$ which satisfy $H_T=\k^{T'}$ for some $T'\subset T$ and which are maximal such, that is, there does not exist $T\subset S\in G_\Gamma$ such that $H_S=\k^{S'}$ with $T\not\subset S'$. 

We claim we must have $H_S=\k^{T'}$ for all $S\in G_\Gamma$ with $S=T'\cup\{v_S\}$ for some $v_S\in V_\Gamma$. Otherwise, let us take such $S$ and consider $H_{T\cup S}$. Since $H_S\neq H_T$, we must have $H_{T\cup S}=H_S+H_T= \k^{T'}+\k w_S$, where $w_S\in \k^{T'}\oplus \k e_{v_S}$. But this implies $V_{T\cup S}=\k^{T'}\oplus \k e_{v_S}=\k^{S}$, contradicting the maximality of $T$.

Running through all such $H_T$, we obtain a collection
 \[
\tilde\tau:=\{T'\colon H_T=\k^{T'} \text{ for  some maximal } T\}.
 \]
We claim that $\tilde\tau\in N(\Gamma)$, so $\tau:=\tilde\tau\sqcup\{V_\Gamma\}\in N^+(\Gamma)$. Suppose that there exist $T_1,T_2\in\tilde\tau$ that are incomparable and satisfy $\Gamma_{T_1\cup T_2}\ne\Gamma_{T_1}\sqcup\Gamma_{T_2}$. Then $H_{T_1\cup T_2}$ contains $\k^{T_1}$ and $\k^{T_2}$. But 
 \[ 
\dim (\k^{T_1}+\k^{T_2}) = |T_1|+|T_2|-|T_1\cap T_2|
 \]
while
 \[
\dim H_{T_1\cup T_2} = |T_1|+|T_2|-|T_1\cap T_2|-1,
 \]
so we have a contradiction, and therefore $\tau:=:\tilde\tau\sqcup\{V_\Gamma\}\in N^+(\Gamma)$. The collection of subspaces $\{H_T\}$ clearly satisfies the first condition defining $X(\Gamma,\tau)$. If it fails to satisfy the second condition, then 
one can find $T\in\tau$, $v_1\ne v_2\in \lambda(T)$, and $T'\subset T\setminus\{v_1,v_2\}$ with 
 \[
T'\in \tau\cup\{\emptyset\}\text{  and } T'\cup\{v_1\}\cup\{v_2\}\in G_\Gamma,
 \] 
so that $H_{T\cup\{v_1\}\cup\{v_2\}}$ coincides with $\k^{T\cup\{v_1\}}$. In that case, we can find a maximal such $T$, in the sense of the previous paragraph, which would imply that we have $T\cup\{v_1\}\in \tau$, which is a contradiction, since in this case $v_1\notin \lambda(T)$. 

We next show that the subsets $X(\Gamma,\tau)$ for various $\tau\in N^+(\Gamma)$ are disjoint, so that
 \[
X(\Gamma)=\bigsqcup_{\tau\in N^+(\Gamma)} X(\Gamma,\tau).
 \] 
Suppose we have a collection $\{H_T\}$ contained in $X(\Gamma,\tau_1)\cap X(\Gamma,\tau_2)$. The argument will depend on whether $\tau_1\cup \tau_2\in N^+(\Gamma)$.

If $\tau_1\cup \tau_2\notin N^+(\Gamma)$, there exist $T_1\in\tau_1$ and $T_2\in\tau_2$ that are incomparable and $\Gamma_{T_1}\sqcup \Gamma_{T_2}\neq\Gamma_{T_1\cup T_2}$. In this case, there are vertices $\{v_1,v_2\}$ such that $v_1\in T_1\setminus T_2$, $v_2\in T_2\setminus T_1$, and $T_1\cup\{v_2\}\in G_\Gamma$, $T_2\cup \{v_1\}\in G_\Gamma$. The first condition for $X(\Gamma,\tau_1)$ implies that $H_{T_1\cup\{v_2\}}=\k^{T_1}$. Note that the first condition defining the subset $X(\Gamma,\tau)$ also forces $H_{T'}=\k^{T'\setminus\{v\}}$ for all $T'\in G_\Gamma$ such that $v\subset T'\subset T\cup\{v\}$. Therefore, $H_{(T_1\cap T_2)\cup\{v_1\}\cup\{v_2\}}=\k^{(G_1\cap G_2)\cup \{v_1\})}$. Similarly, $H_{T_2\cup\{v_1\}}= \k^{T_2}$, and so 
 \[
H_{(T_1\cap T_2)\cup\{v_1\} \cup  \{v_2\}}=\k^{(T_1\cap T_2)\cup\{v_2\})},
 \] 
which is a contradiction.

Suppose that $\tau_1\cup\tau_2\in N^+(\Gamma)$. Let $T\in G_\Gamma$ belong to $\tau_1\setminus\tau_2$. Because of the nested set condition for $\tau_1\cup\tau_2$, there exists at least one vertex $v$ for which $T\cup\{v\}\in G_\Gamma$ and the set $T\cup\{v\}$ is disjoint from every set of $\tau_2$ that does not contain $T$ as a subset. Thus, the first condition for $X(\Gamma,\tau_2)$ implies that $H_{T\cup\{v\}}=\k^{T}$, which means that $T\in \tau_2$, a contradiction.  

Finally, we note that the first condition defining the subset $X(\Gamma,\tau)$ implies that, once the second condition is applicable, the subspace $H_{T\cup\{v_1\}\cup\{v_2\}}$ different from $\k^{T\cup\{v_1\}}$ and from $\k^{T\cup\{v_2\}}$ always contains $\k^{T}$, which gives a $\mathbb{G}_m$ of choices for such subspace. This easily implies that the subsets $X(\Gamma,\tau)$ are toric orbits. In fact, examining the second condition defining the subset $X(\Gamma,\tau)$, it is immediate to describe the stabiliser:
 \[
\mathrm{Stab}(X(\Gamma,\tau))=\left(\prod_{T\in\tau} \mathbb{G}_m\right)/\mathbb{G}_m\subset
\left(\prod_{T\in\tau} \mathbb{G}_m^{\lambda(T)}\right)/\mathbb{G}_m\cong(\mathbb{G}_m)^{V_\Gamma}/\mathbb{G}_m,
 \]
where the product is over all the diagonal inclusions $\mathbb{G}_m\subset\mathbb{G}_m^{\lambda(T)}$.
\end{proof}

It follows from either of the descriptions of our varieties that the closure of each stratum is isomorphic to a product of graphical varieties for smaller graphs, and so inclusions of closed strata provide the collection of all graphical varieties with an operad-like structure. The notion that is instrumental for describing this structure is that of a reconnected complement of a subgraph, which we shall now recall.

\begin{definition}[reconnected complement]
Let $V\in 2^{V_\Gamma}$. The \emph{reconnected complement} of $V$ in $\Gamma$, denoted $\Gamma^*_V$, is the graph obtained from $\Gamma$ by deleting all vertices from $V$ and adding some new edges: specifically, if there is a path in $\Gamma$ that connects two vertices $v_1,v_2\in V_\Gamma\setminus V$ and only uses vertices of $V$ along the way, then there is an edge between $v_1$ and $v_2$ in $\Gamma^*_V$. 
\end{definition}

\begin{figure}[h]
  \centering
  \includegraphics[width=0.9\linewidth]{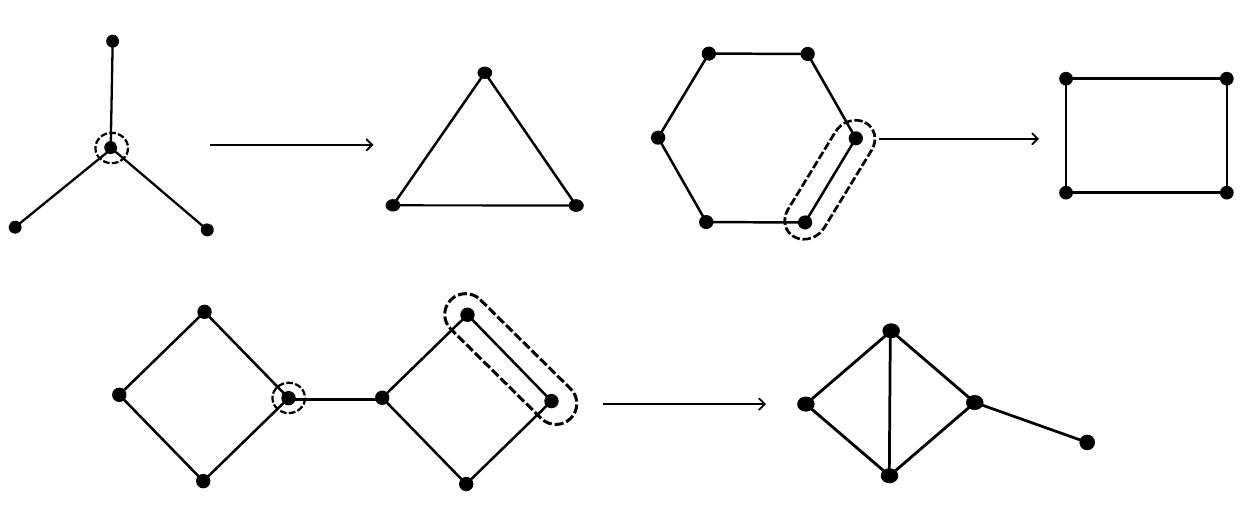}
  \caption{Examples of reconnected complements}
\end{figure}

In the context of graph associahedra, Devadoss and Carr \cite{MR2239078} proved in \cite[Th.~2.9]{MR2239078} that facets of $\calP\Gamma$ are in one-to-one correspondence with elements of $G_\Gamma$ and that the facet corresponding to $V_\Gamma\ne T\in G_\Gamma$ is combinatorially equivalent to the product $\calP\Gamma^*_T\times\calP\Gamma_T$, where $\Gamma^*_T$. From the point of view of toric geometry, this corresponds to the following picture. Among the toric orbits we described above, the orbits of codimension one are the orbits $X(\Gamma,\tau)$ where $\tau=\{T,V_\Gamma\}$ for some $V_\Gamma\ne T\in G_\Gamma$, and we expect that for such $\tau$ we have
 \[
\overline{X(\Gamma,\tau)}\cong X(\Gamma^*_T)\times X(\Gamma_T).
 \]
Let us construct a map 
 \[
\psi_{\Gamma,T}\colon X(\Gamma^*_T)\times X(\Gamma_T)\to X(\Gamma)
 \]
whose image coincides with $\overline{X(\Gamma,\tau)}$. Suppose that 
 \[
\{H'_S\}_{S\in G_{\Gamma^*_T}} \in X(\Gamma^*_T),\quad \{H''_S\}_{S\in G_{\Gamma_T}} \in X(\Gamma_T).
 \] 
We define the image of this pair as the element
$\{H_S\}_{S\in G_{\Gamma}} \in X(\Gamma)$ defined by the formula
\begin{equation}\label{eq:composition-toric}
H_S = 
\begin{cases}
H'_{S \setminus T} \oplus \k^{S\cap T}, \text{ if } S \not\subset T\\
\qquad H''_S, \qquad \text{\, if } S \subset T.
\end{cases}
\end{equation}
By a direct inspection, the image of this map is the closed stratum $\overline{X(\Gamma,\tau)}$. Considering how smaller closed strata are obtained of iterations of such maps is precisely what leads one to the definition of a reconnectad in the following section. Let us note that, though this last remark used the reconnected complement of $T\in G_\Gamma$, the definition is available for any $V\in 2^{V_\Gamma}$, not necessarily belonging to $G_\Gamma$, and we shall later use it in full generality. 

\begin{remark}
The combinatorics of reconnected complements was used in the recent work of Forcey and Ronco \cite{https://doi.org/10.1112/jlms.12596} to define a strict operadic category in the sense of Batanin and Markl \cite{MR3406537}. The formalism we develop below is very close to that one, except for two important differences. Firstly, we do not require the set of vertices of a graph to carry a linear order, and when we impose a linear order to define the related ``shuffle'' formalism, this will constrain the existing morphisms. Secondly, operads for the operadic category of \cite{https://doi.org/10.1112/jlms.12596} are defined by taking as the starting point the reconnected complements for $T\in G_\Gamma$, and then building all possible composition maps as composites of these. Our approach will arrive at these operations from two other definitions which exhibit clearly defined categorical constructions that are not apparent at a first glance for the approach of \cite{https://doi.org/10.1112/jlms.12596}.
\end{remark}

\section{Reconnectads}\label{sec:reconnectads}

In this section, we shall define and study a new operad-like structure responsible for stratifications of graphical varieties by toric orbits. We begin with the definition of a graphical collection, generalising the notion of a species of structures valued in the category $\sfC$, see \cite{MR1629341,MR633783}.

\begin{definition}[graphical collection]
The \emph{groupoid of connected graphs} $\CGr$ is the category whose objects are connected simple graphs and whose morphisms are graph isomorphisms.
A \emph{graphical collection} with values in $\sfC$ is a functor $\calF\colon \CGr\to\sfC$ satisfying $\calF(\emptyset)=1_{\sfC}$. All graphical collections with values in $\sfC$ form a category $\GrCol_{\sfC}$, where morphisms are natural transformations. 
\end{definition}

\subsection{The monad of nested sets}

Operads can be thought of as algebras over the monad of trees. We shall now define the monad of nested sets on the category of graphical collections. 

\begin{definition}[nested set endofunctor]
The \emph{nested set endofunctor} $\calN$ on the category of graphical collections is defined as follows. Let $\calX$ be a graphical collection. The graphical collection $\calN(\calX)$ has the components
 \[
\calN(\calX)(\Gamma):=\bigoplus_{\tau\in N^+(\Gamma)}\bigotimes_{T\in \tau}\calX((\Gamma_T)^*_{T\setminus\lambda(T)}).
 \] 
\end{definition}

Let us explain how to give $\calN$ a natural structure of a monad. The unit is the natural inclusion $\calX\to\calN(\calX)$ corresponding to $\tau=\{V_\Gamma\}$. The natural maps
 \[
\calN(\calN(\calX))\to\calN(\calX)
 \]
come from the slogan saying that ``a nested set of nested sets is a nested set''. More precisely, suppose that $\tau\in N^+(\Gamma)$. We would like to define a map 
 \[
\bigotimes_{T\in \tau}\calN(\calX)((\Gamma_T)^*_{T\setminus\lambda(T)})\to \calN(\calX)(\Gamma). 
 \]
For that, we note that the left hand side is the sum of tensor products over all possible nested sets of connected graphs $(\Gamma_T)^*_{T\setminus\lambda(T)}$ on the vertex sets $\lambda(T)$, and to each such collection of nested sets one can canonically associate a nested set of $\Gamma$ by joining together the subsets which violate the condition $\Gamma_{U\cup V}=\Gamma_U\sqcup\Gamma_V$, thus obtaining a map of the required form. (In terms of trees $\mathbb{T}_\tau$ associated to nested sets, this corresponds to grafting of trees, which can be used to establish the associativity required by the monad axioms.)

In particular, for every $\tau=\{T,V_\Gamma\}$ with $V_\Gamma\ne T\in G_\Gamma$, the corresponding tree $\mathbb{T}_\tau$ has $\lambda(T)=T$ and $\lambda(V_\Gamma)=V_\Gamma\setminus T$, and moreover,
 \[
(\Gamma_T)^*_{\emptyset}=\Gamma_T, \quad (\Gamma_{V_\Gamma})^*_{V_\Gamma\setminus (V_\Gamma\setminus T)}=(\Gamma)^*_T,
 \]
so we find a summand $\calX(\Gamma^*_T)\otimes\calX(\Gamma_T)$ in $\calN(\calX)$. This observation will shine through in the following sections; for now we use this particular kind of nested sets in an example.

\begin{example}
Let us consider the graph $\Gamma$ depicted in the top left corner of Figure~\ref{fig:nested set composition}. We choose the subset $T=\{1,2,3\}$ for which the induced graph $\Gamma_T$ is isomorphic to $K_3$ and the reconnected complement $\Gamma^*_T$ is isomorphic to $P_3$. 
\begin{figure}[ht]
    \centering
    \def\svgwidth{\columnwidth}
    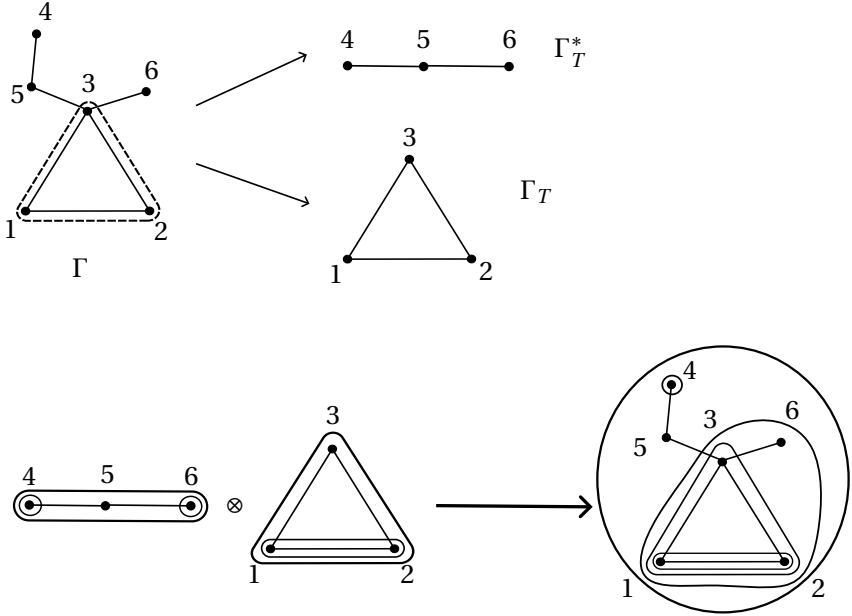
    \caption{Nested sets of nested sets and compositions}
    \label{fig:nested set composition}
\end{figure}
For the nested set $\tau=\{T,V_\Gamma\}$, there is a term in $\calN(\calN(\calX))$ corresponding to the tensor product of the term of $\calN(\calX)(\Gamma^*_T)$ associated to the nested set $\{\{4\},\{6\},\{4,5,6\}\}$ and the term of $\calN(\calX)(\Gamma_T)$ associated to the nested set $\{\{1,2\},\{1,2,3\}\}$. Joining these together gives us the nested set $\{\{4\},\{1,2\},\{1,2,3\},\{1,2,3,6\},\{1,2,3,4,5,6\}\}$ of $\Gamma$ depicted in the bottom right corner of Figure~\ref{fig:nested set composition}. 
\end{example}

\begin{definition}[reconnectad, monadic definition]
A \emph{reconnectad} is an algebra over the nested set monad. Concretely, it is a graphical collection $\calF$ equipped with structure maps
 \[
\calN(\calF)\to\calF
 \]
compatible with the monad structure on $\calN$ in the usual sense. 
\end{definition}

Our definition of a reconnectad leads to an immediate definition of the free reconnectad generated by a graphical collection $\calX$. 

\begin{definition}[free reconnectad]
The \emph{free reconnectad} generated by a graphical collection $\calX$ is the graphical collection $\calN(\calX)$ with the structure maps
 \[
\calN(\calN(\calX))\to\calN(\calX)
 \]
defined above to encode the monad structure.
\end{definition}

The following proposition, which will be very useful for us, is straightforward from the definition.

\begin{proposition}\label{prop:grEnd-monoidal}
Suppose that $\sfC$ and $\sfC'$ are two symmetric monoidal categories satisfying all the assumptions we impose on a symmetric monoidal category, and that $\phi\colon \sfC\to \sfC'$ is a symmetric monoidal functor. Then for each reconnectad $\calF$ in $\sfC$, we obtain a reconnectad $\phi(\calF)$ in $\sfC'$. In particular, for a reconnectad $\calF$ in topological spaces, the graphical collection $H_{\bullet}(\calF,\k)$ is a reconnectad in $\dgVect$.
\end{proposition}

Our definition of a reconnectad is not very easy to unwrap to handle various particular cases. We shall now present two equivalent definitions that will allow us to view reconnectads in a much more concrete way. 

\subsection{The monoidal category of graphical collections}

There is an obvious symmetric monoidal structure on the category of graphical collections called the \emph{Hadamard product}. It is defined by the formula 
 \[
(\calF\underset{H}{\otimes}\calG)(\Gamma):=\calF(\Gamma)\otimes\calG(\Gamma).
 \]
The graphical collection $\mathbb{I}$ with $\mathbb{I}(\Gamma)=1_{\sfC}$, equipped with the trivial action of the group $\Aut(\Gamma)$, is the unit of this monoidal structure. For our purposes, another (highly non-symmetric) monoidal structure on the category of graphical collections will play an important role. It is defined as follows.

\begin{definition}[reconnected product]
The \emph{reconnected product} of two graphical collections $\calF$ and $\calG$ is the graphical collection $\calF\circ_{\rmR}\calG$ defined by the formula
 \[
(\calF\circ_{\rmR}\calG)(\Gamma):=\bigoplus_{V\subset V_\Gamma}\calF(\Gamma^*_V)\otimes\bigotimes_{\Gamma'\in\mathrm{Conn}(\Gamma_V)} \calG(\Gamma').
 \]
\end{definition}

It turns out that the product we defined makes the category of graphical collections into a monoidal category. 

\begin{proposition}\label{prop:assoc}
The category $\GrCol_{\sfC}$ equipped with the reconnected product $\circ_{\rmR}$ is a monoidal category whose unit is the graphical collection $\mathbbold{1}$ defined by 
 \[
\mathbbold{1}(\Gamma)=
\begin{cases}
1_{\sfC}, \,\Gamma=\emptyset,\\
\,\,0, \,\,\,\text{ otherwise.}
\end{cases}
 \]
\end{proposition}

\begin{proof}
Let us allow ourselves to evaluate each graphical collection on not necessarily connected graphs by putting 
 \[
\calG(\Gamma_1\sqcup\Gamma_2):=\calG(\Gamma_1)\otimes\calG(\Gamma_2).
 \]
This permits us to write the definition of the reconnected product in a more compact way
 \[
(\calF\circ_{\rmR}\calG)(\Gamma)=\bigoplus_{V\subset V_\Gamma}\calF(\Gamma^*_V)\otimes\calG(\Gamma_V).
 \]
Notice that this convention does not lead to a contradiction when evaluating the reconnected product on disconnected graphs: we have 
\begin{multline*}
(\calF\circ_{\rmR}\calG)(\Gamma_1\sqcup\Gamma_2)=
(\calF\circ_{\rmR}\calG)(\Gamma_1)\otimes (\calF\circ_{\rmR}\calG)(\Gamma_2)=\\
\bigoplus_{V_1\subset V_{\Gamma_1}}\calF((\Gamma_1)^*_{V_1})\otimes\calG((\Gamma_1)_{V_1})\otimes
\bigoplus_{V_2\subset V_{\Gamma_2}}\calF((\Gamma_2)^*_{V_2})\otimes\calG((\Gamma_2)_{V_2}),
\end{multline*}
which, thanks to the symmetry isomorphisms of $\sfC$, the property $\calG(\emptyset)=1_{\sfC}$, and the properties
\begin{gather*}
(\Gamma_1\sqcup\Gamma_2)^*_{V}=(\Gamma_1)^*_{V\cap V_{\Gamma_1}}\sqcup (\Gamma_2)^*_{V\cap V_{\Gamma_2}},\\
(\Gamma_1\sqcup\Gamma_2)_{V}=(\Gamma_1)_{V\cap V_{\Gamma_1}}\sqcup (\Gamma_2)_{V\cap V_{\Gamma_2}},
\end{gather*}
is isomorphic to 
 \[
\bigoplus_{V\subset V_{\Gamma_1\sqcup\Gamma_2}}\calF((\Gamma_1\sqcup\Gamma_2)^*_{V})\otimes\calG((\Gamma_1\sqcup\Gamma_2)_{V}).
 \] 
Because of that, we have
\begin{multline*}
(\calF\circ_{\rmR}(\calG\circ_{\rmR}\calH))(\Gamma)=
\bigoplus_{V\subset V_\Gamma}\calF(\Gamma^*_V)\otimes(\calG\circ_{\rmR}\calH)(\Gamma_V)\\=
\bigoplus_{V\subset V_\Gamma}\calF(\Gamma^*_V)\otimes\bigoplus_{U\subset V}\calG((\Gamma_V)^*_U)\otimes\calH((\Gamma_V)_U)
\end{multline*}
and
\begin{multline*}
((\calF\circ_{\rmR}\calG)\circ_{\rmR}\calH)(\Gamma)=
\bigoplus_{U\subset V_\Gamma}(\calF\circ_{\rmR}\calG)(\Gamma^*_U)\otimes\calH(\Gamma_U)\\
=\bigoplus_{U\subset V_\Gamma}\bigoplus_{W\subset V_\Gamma\setminus U}
\calF((\Gamma^*_U)^*_W)\otimes\calG((\Gamma^*_U)_W)\otimes \calH(\Gamma_U),
\end{multline*}
which, if we denote $V:=U\sqcup W$, becomes
 \[
\bigoplus_{U\subset V_\Gamma}\bigoplus_{U\subset V\subset V_\Gamma}
\calF((\Gamma^*_U)^*_{V\setminus U})\otimes\calG((\Gamma^*_U)_{V\setminus U})\otimes \calH(\Gamma_U), 
 \]
and the associativity isomorphism follows from the obvious properties 
 \[
(\Gamma^*_U)^*_{V\setminus U}=\Gamma^*_V, \quad (\Gamma^*_U)_{V\setminus U}=(\Gamma_V)^*_U, \quad (\Gamma_V)_U=\Gamma_U
 \]
that hold for any $U\subset V\subset V_\Gamma$. Moreover, the associativity isomorphisms are immediately seen to satisfy the axioms of a monoidal category. Finally, we note that we have $\Gamma_\emptyset=\emptyset$ and $\Gamma^*_\emptyset=\Gamma$ and, dually, we have $\Gamma_{V_\Gamma}=\Gamma$ and $\Gamma^*_{V_\Gamma}=\emptyset$. This immediately implies the isomorphisms 
 \[
\calF\circ_{\rmR}\mathbbold{1}\cong\mathbbold{1}\circ_{\rmR}\calF\cong\calF,
 \] 
and the necessary compatibility of these isomorphisms with the monoidal structure is verified by direct inspection. 
\end{proof}

This result ensures that the following definition of a reconnectad makes sense; it is straightforward to see that it is equivalent to the monadic definition.

\begin{definition}[reconnectad, monoidal definition]
A \emph{reconnectad} is a monoid in the monoidal category of graphical collections equipped with the reconnected product $\circ_{\rmR}$.
\end{definition}

For a reconnectad $\calF$, a connected graph $\Gamma$, and a subset $V$ of $V_\Gamma$, we shall denote by $\mu_V^\Gamma$ the restriction of the structure map $(\calF\circ_{\rmR}\calF)(\Gamma)\to \calF(\Gamma)$ to the summand $\calF(\Gamma^*_V)\otimes\bigotimes_{\Gamma'\in\mathrm{Conn}(\Gamma_V)} \calF(\Gamma')$. In the particular case when $V_\Gamma\ne T\in G_\Gamma$, we shall use the notation $\circ_T^\Gamma$ for $\mu_T^\Gamma$; this distinction should help the reader to navigate between the connected and the disconnected situation. Additionally, when $V=\{v\}$, we write $\circ_v^\Gamma$ instead of $\circ_{\{v\}}^\Gamma$ to simplify the notation slightly. 

One immediate consequence of the monoidal definition is that whenever one can talk about monoids, one can also talk about comonoids, and so the notion of a \emph{coreconnectad} arises naturally. We shall denote by $\Delta_V^\Gamma$ the composition of the structure map $\calG(\Gamma)\to (\calG\circ_{\rmR}\calG)(\Gamma)$ of a coreconnectad $\calG$ with the projection onto the summand $\calG(\Gamma^*_V)\otimes\bigotimes_{\Gamma'\in\mathrm{Conn}(\Gamma_V)} \calG(\Gamma')$ of $(\calG\circ_{\rmR}\calG)(\Gamma)$.

\subsection{The coloured operad encoding reconnectads}

In the case of operads, on can use the operations $\circ_i$, often referred to as infinitesimal compositions, or partial compositions, to give an equivalent definition. Even though this viewpoint somewhat obscures the fact that operads are associative monoids, it allows one to view operads as algebras over a coloured operad, which has its advantages.

The following proposition is proved by direct inspection (using the properties of restrictions and reconnected complements used in the proof of Proposition \ref{prop:assoc}).

\begin{proposition}\label{prop:infini}
The datum of a reconnectad on a graphical collection $\calF$ is equivalent to the datum of infinitesimal compositions
 \[
\circ^\Gamma_T\colon\calF(\Gamma^*_T)\otimes\calF(\Gamma_T)\to\calF(\Gamma)
 \]
for all $V_\Gamma\ne T\in G_\Gamma$; these operations must satisfy the following properties:
\begin{itemize}
    \item (unit axiom) Under the identifications 
     \[
\Gamma^*_{\emptyset}=\Gamma,\quad \Gamma^*_{V_\Gamma}=\emptyset,\quad 1_{\sfC}\otimes X\cong X\otimes1_{\sfC}\cong X, 
     \]
    we have $\circ^\Gamma_\emptyset=\circ^\Gamma_{V_\Gamma}=\mathrm{id}_{\calF(\Gamma)}$.
    \item (parallel axiom) For all $T_1,T_2\in G_\Gamma$ such that $\Gamma_{T_1\cup T_2}=\Gamma_{T_1}\sqcup \Gamma_{T_2}$, the diagram
\[\begin{tikzcd}
\calF(\Gamma^*_{T_1\cup T_2}) \otimes \calF(\Gamma_{T_1}) \otimes \calF(\Gamma_{T_2}) \arrow{r}{\circ_{T_1}^{\Gamma^*_{T_2}}} \arrow[swap]{d}{\circ_{T_2}^{\Gamma^*_{T_1}}} & \calF(\Gamma^*_{T_2}) \otimes \calF(\Gamma_{T_2}) \arrow{d}{\circ_{T_2}^{\Gamma}} \\
\calF(\Gamma^*_{T_1}) \otimes \calF(\Gamma_{T_1})\arrow{r}{\circ_{T_1}^{\Gamma}} & \calF(\Gamma)
\end{tikzcd}
\]
    commutes. 
    \item (consecutive axiom) For all $T_1,T_2\in G_\Gamma$ with $T_1\subset T_2$, the diagram
\[\begin{tikzcd}
\calF(\Gamma^*_{T_2})\otimes \calF((\Gamma_{T_2})^*_{T_1}) \otimes \calF(\Gamma_{T_1}) \arrow{r}{\mathrm{id} \otimes \circ_{T_1}^{\Gamma_{T_2}}} \arrow[swap]{d}{\circ_{T_2\setminus T_1}^{\Gamma^*_{T_1}}\otimes \mathrm{id}} & \calF(\Gamma^*_{T_2}) \otimes\calF(\Gamma_{T_2}) \arrow{d}{\circ_H^{\Gamma}} \\
\calF(\Gamma^*_{T_1}) \otimes \calF(\Gamma_{T_1})  \arrow{r}{\circ_{T_1}^{\Gamma}} & \calF(\Gamma)
\end{tikzcd}
\]
    commutes.
    \item(equivariance) For every $T\in G_\Gamma$ and every automorphism $\alpha \in \mathrm{Aut}(\Gamma)$ the diagram 
\[\begin{tikzcd}
\calF(\Gamma^*_{T}) \otimes \calF(\Gamma_{T}) \arrow{r}{\circ_T^{\Gamma}} \arrow[swap]{d}{\phi} & \calF(\Gamma) \arrow{d}{\phi} \\
\calF(\Gamma^*_{\alpha(T)}) \otimes \calF(\Gamma_{\alpha(T)}) \arrow{r}{\circ_{\alpha(T)}^{\Gamma}} & \calF(\Gamma)
\end{tikzcd}
\]
    commutes.
\end{itemize}
\end{proposition}

This proposition immediately implies that the maps $\psi_{\Gamma,T}$ defined by Formula~\eqref{eq:composition-toric} give the collection of all toric varieties of graph associahedra the structure of a reconnectad, thus introducing the central example of a reconnectad that motivated our work.

\begin{definition}[wonderful reconnectad]
The \emph{wonderful reconnectad} is the graphical collection $\calW$ with
 \[
\calW(\Gamma):=X(\Gamma)
 \]
and with the structure operations 
 \[
\circ^\Gamma_T\colon\calW(\Gamma^*_T)\otimes\calW(\Gamma_T)\to\calW(\Gamma)
 \]
given by $\circ^\Gamma_T:=\psi_{\Gamma,T}$. 
\end{definition}

We shall use Proposition \ref{prop:infini} in conjunction with the notion of a groupoid coloured operad of Petersen \cite{MR3134040} as follows, mimicking the approach to modular operads of Ward \cite{MR4425832}, see also \cite{DSVV}. 

\begin{definition}\label{def:groupoidcoloured}
We define the $\CGr$-coloured operad $\Rec=\calT(E)/(R)$ as follows. It is generated by the elements 
\[
 E\big(\Gamma; \Gamma^*_T, \Gamma_T\big) :=\\ 
 \left\{
\Aut(\Gamma^*_T)\times \Aut(\Gamma_T)\times
   \begin{aligned}\begin{tikzpicture}[optree]
    \node{}
      child { node[circ]{$\circ^\Gamma_T$}
        child { edge from parent node[left,near end]{\tiny$\Gamma^*_T$} }
        child { edge from parent node[right,near end]{\tiny$\Gamma_T$} }
        edge from parent node[right,near start]{\tiny$\phantom{i}\Gamma$}
         } ;
   \end{tikzpicture}\end{aligned} \right\}\ ,  V_\Gamma\ne T\in G_\Gamma,
\]
with the regular $\Aut(\Gamma^*_V)\times \Aut(\Gamma_V)$-action and with the $\Aut(\Gamma)$-action given by
\[\left(\circ_V^\Gamma,\mathrm{id},\mathrm{id}\right)^{\phi}=\left(\circ_{\phi(V)}^\Gamma,\phi|_{\Gamma^*_V},\phi|_{\Gamma_V}\right)\ .\]
   The quadratic relations $R$ are the ones given in Proposition \ref{prop:infini}.
\end{definition}

We are now in the position to give another equivalent definition of a reconnectad.

\begin{definition}
A reconnectad is an algebra over the $\CGr$-coloured operad $\Rec$.
\end{definition}

One immediate consequence of this definition is that all the standard constructions for algebras over operads are available for reconnectads. It is also useful to note that the $\CGr$-coloured operad $\Rec$ has an obvious diagonal making it a Hopf $\CGr$-coloured operad. In particular, the category of reconnectad has a symmetric monoidal structure: that structure is the Hadamard product equipped with the obvious composition maps.

\begin{remark}\label{rem:Feynman}
Expressing certain operadic structures as algebras over groupoid coloured operads is essentially equivalent to talking about operads over a Feynman category \cite{MR3636409}. Let $\Gr$ denote the category whose objects are simple (not necessarily connected) graphs and whose morphisms are generated by graph isomorphisms and the morphisms
 \[
\phi^\Gamma_V \colon \Gamma_V\sqcup\Gamma^*_V  \to \Gamma
 \]
that are associated to the datum of a graph $\Gamma$ and a choice of a subset $V\subset V_\Gamma$. The groupoid of connected graphs $\CGr$ is a full subcategory of $\Gr$. By a direct inspection, the triple $(\CGr,\Gr,\imath)$, where $\imath$ is the inclusion $\CGr\to\Gr$ defines the datum of a Feynman category, and reconnectads are operads over this Feynman category.  
\end{remark}

\subsection{Particular types of graphs}\label{sec:part-types}

If we restrict ourselves to various families of graphs, we may recognize known algebraic structures in the guise of reconnectads. 

Recall that a twisted associative algebra is a symmetric collection (a functor from the groupoid of finite sets to $\sfC$) which is a monoid with respect to the Cauchy monoidal structure 
 \[
(\calF\cdot\calG)(I):=\bigoplus_{I=J\sqcup K}\calF(J)\otimes\calG(K). 
 \]
on symmetric collections. A twisted associative algebra $\calA$ is said to be connected if $\calA(\emptyset)=1_{\sfC}$.

\begin{proposition}\label{prop:LM}
Suppose that we restrict ourselves to the full subcategory of collections supported on complete graphs. The datum of a reconnectad in that category is the same as the datum of a connected twisted associative algebra.
\end{proposition}

\begin{proof}
Note that the datum of a graphical collection supported on complete graphs is obviously the same as the datum of a symmetric collection whose evaluation on the empty set is $1_{\sfC}$: indeed, a complete graph carries as much information as its set of vertices. Moreover, for a complete graph $\Gamma$ and every $V\subset V_\Gamma$, the graph $\Gamma_V$ is connected and complete, and the graph $\Gamma^*_V$ is also complete, so the monoidal structure on our category corresponds precisely to the Cauchy monoidal structure. 
\end{proof}

Recall that a nonsymmetric operad is a nonsymmetric collection (a functor from the groupoid of finite totally ordered sets to $\sfC$) which is a monoid with respect to the composition product 
 \[
(\calF\circ\calG)(I):=\bigoplus_{k\ge 0}\bigoplus_{I=I_1+\cdots+I_k}\calF(\{1,\ldots,k\})\otimes\calG(I_1)\otimes\cdots\otimes G(I_k). 
 \]
on nonsymmetric collections. A nonsymmetric operad $\calO$ is said to be reduced if $\calO=0$ and connected if $\calO(\{\mathrm{pt}\})=1_{\sfC}$. Let us call a nonsymmetric operad \emph{mirrored} if it comes from a functor from the groupoid quotient by the $\mathbb{Z}/2\mathbb{Z}$-action reversing the order. Components of such an operad have actions of $\mathbb{Z}/2\mathbb{Z}$ for which for all elements $f$ and $g$ of arities $p$ and $q$ respectively, we have
 \[
\sigma(f\circ_i g)=\sigma(f)\circ_{p-i+1}\sigma(g). 
 \]

\begin{proposition}\label{prop:NS}
Suppose that we restrict ourselves to the full subcategory of collections supported on path graphs (that is, on the type $A$ Dynkin diagrams). The datum of a reconnectad in that category is the same as the datum of a connected reduced mirrored nonsymmetric operad.
\end{proposition}

\begin{proof}
 Let us consider the assignment to each nonempty finite totally ordered set $I$ the set of \emph{gaps} 
 \[
\mathrm{Gap}(I)=\{(i_1,i_2)\colon i_1,i_2\in I, \{i\in I \colon i_1<i<i_2\}=\emptyset\}.
 \]
The identification of reconnectads supported on path graphs with nonsymmetric operads is conveniently described using gaps: if $\calF$ is a reconnectad, we may define a nonsymmetric collection $\calO$ supported on non-empty finite totally ordered sets by the rule
 \[
\calO(I):=\calF(\Gamma_I),
 \]
where the path graph $\Gamma_I$ has the vertex set $\mathrm{Gap}(I)$, and two pairs are connected by an edge if they share a vertex. If $\calO$ comes from a functor from the groupoid quotient by the $\mathbb{Z}/2\mathbb{Z}$-action reversing the order, this rule really defines a graphical collection compatible with the automorphisms of path graphs. Noting that for a path graph $\Gamma$ and every $V\subset V_\Gamma$ for which the graph $\Gamma_V$ is connected, the graph $\Gamma_V$ is also a path graph, and the graph $\Gamma^*_V$ is also a path graph, we see  by direct inspection that under our identification the composite product of nonsymmetric collections corresponds precisely to the reconnected product of graphical collections. 
\end{proof}

The notion of a twisted associative algebra, as explained in \cite{MR3642294}, is a symmetric version of the notion of a permutad of Loday and Ronco \cite{MR2995045} who in turn refer to permutads as a ``noncommutative version of nonsymmetric operads''. We now see that the universe of reconnectads is where the two notions meet in a meaningful way.

It would be interesting to consider the algebraic structures corresponding to other families of graphs that are closed under the operations $\Gamma_V$ and $\Gamma^*_V$. Two very interesting examples are the family including all complete graphs and all stellar graphs (see Figure~\ref{fig:stellar}), which is already featured in a prominent way in \cite{Stellar}, and the family including all path graphs and all cycle graphs, which should be examined in the context of the notion of a noncommutative cohomological field theory \cite{MR4072173}.

\begin{figure}[h]
  \centering
  \includegraphics[width=0.7\linewidth]{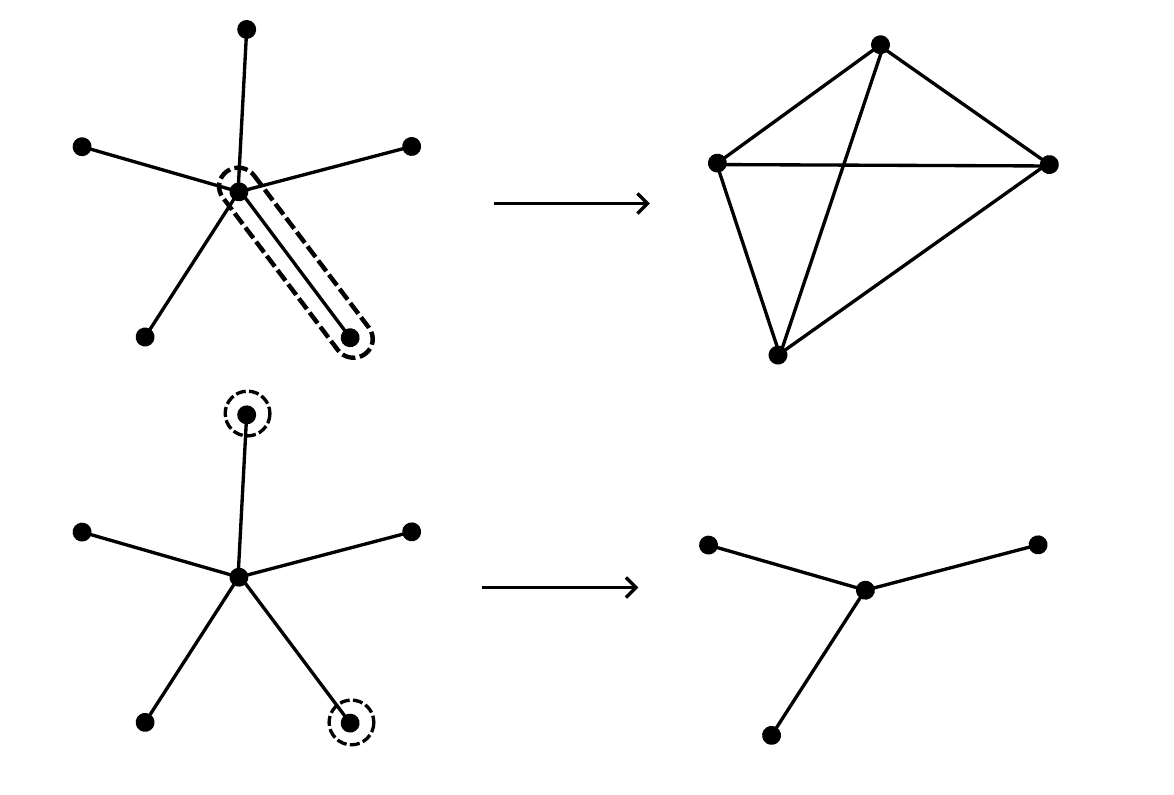}
  \caption{Reconnected complements of a stellar graph}
  \label{fig:stellar}
\end{figure}

\begin{remark}
Reconnectads themselves can be viewed as a restriction of a much more general formalism of operads over a Feynman category of built lattices developed in the recent work of B.~Coron \cite{CoronLattices} for all geometric lattices and their building sets. However, some pleasant geometric aspects of the context we work in, particularly the underlying toric geometry, are not available in that generality. 
\end{remark}

\subsection{The commutative reconnectad}

The simplest possible example of a reconnectad is the terminal reconnectad in the category of sets, which we shall refer to as the \emph{commutative reconnectad} because of its similarity with the operad of commutative associative algebras. Following the notational pattern chosen in \cite{DST}, we denote this reconnectad by $\grCom$, intenting the letters `g' and `r' to remind the reader of the words ``graphs'' and ``reconnected''. This reconnectad has the components 
 \[
\grCom(\Gamma)=1_{\sfC}, 
 \]
and each structure operation arises from the monoid unit property for $1_{\sfC}$. The reconnectad axioms are trivially satisfied. We already saw this graphical collection as the unit $\mathbb{I}$ of the Hadamard product.

The construction of the reconnectad $\grCom$ can be naturally generalised as follows.

\begin{definition}
Let $X$ be an object of the category $\sfC$. We define the \emph{commutative reconnectad of $X$} , denoted by $\grCom_X$, by putting 
 \[
\grCom_X(\Gamma)=\bigotimes_{v\in V_\Gamma} X, 
 \]
with the structure operations   
 \[
\mu^\Gamma_V\colon\grCom_X(\Gamma^*_V)\otimes\grCom_X(\Gamma_V)\to\grCom_X(\Gamma)
 \]
coming from the isomorphisms 
 \[
\bigotimes_{v\in V=V_{\Gamma_V}} X\otimes \bigotimes_{v\in V_\Gamma\setminus V=V_{\Gamma^*_V}} X\cong \bigotimes_{v\in V_\Gamma} X.
 \]
\end{definition}
The reconnectad axioms are trivially satisfied for $\grCom_X$. 

\section{Algebraic constructions for reconnectads}\label{sec:algebra-reconnectads}

In this section, we shall develop the necessary algebraic formalism to work with reconnectads in the linearised context, so we assume that the symmetric monoidal category $\sfC$ is the category $\dgVect$ of chain complexes (or its full subcategory of homologically graded vector spaces, viewed as chain complexes with zero differential, or the full subcategory of ungraded vector spaces, viewed as chain complexes concentrated in homological degree zero).  

Since a lot of results of this section rely on good understanding of free reconnectads, let us remind the reader that, within our approach, elements of the free reconnectad are linear combinations of elements corresponding to additional decorations of trees $\mathbb{T}_\tau$ associated to nested sets: each vertex $T$ of such tree is decorated by an element of $\calX((\Gamma_T)^*_{T\setminus\lambda(T)})$. There are many other ways to view the free reconnectad: for instance, one may use the general result of \cite{MR2471247} on free monoids in monoidal categories, or the general formalism of \cite{MR3636409} that constructs the free operad for the given Feynman category. It is almost immediate that our construction of the free reconnectad is isomorphic to those other ones; we choose it since it can be made sufficiently concrete but yet has certain categorical elegance. 

\subsection{Presentations by generators and relations}

The notion of an ideal is available for reconnectads in the category of vector spaces. In general, the notion of ideal should be given in a way that the first isomorphism theorem holds: ideals are kernels of surjective morphisms of reconnectads. Precisely, an \emph{ideal} in a reconnectad $\calF$ is a graphical subcollection $\calI\subset\calF$ for which every structure map $\mu_V^\Gamma$, when evaluated on elements of $\calF(\Gamma^*_V)\otimes\calF(\Gamma_V)$ where at least one tensor factor is in $\calI$, assumes values in $\calI$.

Using ideals of free reconnectads, we may talk about presentations of reconnectads by generators and relations.  A reconnectad $\calF$ is \emph{presented by generators $\calX$ and relations $\calR\subset\calN(\calX)$} if it is isomorphic to the quotient of the free reconnectad $\calN(\calX)$ by the ideal $\langle\calR\rangle$ generated by $\calR$. Let us give some examples of presentations by generators and relations. 

\begin{proposition}\label{prop:relgrcom}
The reconnectad $\grCom_X$ is generated by the graphical collection $\calX$ supported at $P_1$ for which $\calX(P_1)=X$, and the defining relations are 
 \[
\mu^{P_2}_{1}(x_1\otimes x_2)=\mu^{P_2}_{2}(\sigma(x_1\otimes x_2)).
 \]
\end{proposition}

\begin{proof}
First, we note that the component $X=\grCom_X(P_1)$ generates the reconnectad $\grCom_X$, since for each graph $\Gamma$ and each vertex $v\in V_\Gamma$, we can consider the set $S=V_\Gamma\setminus\{v\}$, and already the map 
 \[
\mu^\Gamma_S\colon \grCom_X(\Gamma^*_S)\times\grCom_X(\Gamma_S)\to\grCom_X(\Gamma) 
 \]
is surjective, since $\Gamma^*_S=\{v\}$, and we may argue by induction on $|V_\Gamma|$. Let us show that the given relations are sufficient to present the reconnectad $\grCom_X$. For that, we note that the given relations and the parallel axioms of a reconnectad allow one to transform any iterated composition into one defined by a chosen spanning tree of $\Gamma$ and a particular way of ``assembling'' that tree by adding vertices one by one while keeping the intermediate graph connected.
\end{proof}

\subsection{Bar-cobar duality and Koszul duality}\label{sec:dualities}

In the case of operads, the ``true'' reason behind the bar-cobar duality and the closely related to it Koszul duality comes from the fact that the coloured operad encoding operads happens to be Koszul. In this section, we shall see that the situation with reconnectads is completely analogous. 

\subsubsection{Koszulness of the coloured operad $\texorpdfstring{\Rec}{Rec}$}

\begin{proposition}
The groupoid-coloured operad $\Rec$ is Koszul. 
\end{proposition}

\begin{proof}
We shall use the relationship to Feynman categories discussed in Remark \ref{rem:Feynman}, assuming a certain fluency in the language of Feynman categories \cite{MR3636409}. Let us define a degree function $\deg$ on morphisms of $\Gr$ as follows. We set $\deg(f)=0$ if $f$ is an isomorphism, $\deg(f)=1$ if $f=\phi^\Gamma_T$ for $T\in G_\Gamma$, and then imposing
$\deg(f\sqcup g)=\deg(f)+ \deg(g)$
and
$\deg(f\circ g)=\deg(f)+\deg(g)$. 
It is immediate to see that the degree is well defined and is a proper degree function in the sense of \cite[Def.~7.2.1]{MR3636409}. If we consider the set $C_n(X,Y)$ of all chains of $n$ or more composable morphisms of $\Gr$ of which exactly $n$ have non-zero degree, such that the composition is a morphism $X\to Y$, modulo the equivalence relation induced by composing degree zero morphisms, and the subset $C^+_n(X,Y)$ the subset of $C_n(X,Y)$ consisting of chains of morphisms of degree at most one. Using the parallel and consecutive axioms for infinitesimal compositions, we obtain a free action of $S_n$ on $C_n(X,Y)$. Specifically, an equivalence class in $C_n(X,Y)$ may be identified with a sequence of $n$ morphisms of non-zero degree. Each such morphism corresponds to a nested set of a graph, and for every pair of adjacent morphisms, we have a unique corresponding commutative diagram, induced by those of simple reduction maps. We define the action of the transposition $(i\ i+1)$ on a chain $(f_1,\ldots,f_n)$ to be the map sending the chain to the chain obtained by replacing morphisms  $f_i,f_{i+1}$ by the unique morphisms $g_i,g_{i+1}$ coming from the associated commutative diagram. More explicitly, we must have that $f_i,f_{i+1}$ correspond to a 2-step reconnected complement of a graph $\Gamma$; performing them in the opposite order (in the appropriate sense) leads to the choice of $g_i,g_{i+1}$. This action is clearly free, and $C^+_n(X,Y)_{S_n}\cong \Hom_n(X,Y)$. Thus, implies that, in the terminology of \cite[Def.~7.2.2]{MR3636409}, the Feynman category $\Gr$ is cubical. By the main result of \cite{KWKoszul}, this Feynman category is Koszul. Examining the bar construction of the Feynman category $\Gr$, we find exactly the bar construction of the groupoid coloured operad $\Rec$, and therefore the latter operad is Koszul.   
\end{proof}

This proposition ensures that there are reconnectad analogues of the bar-cobar duality and the Koszul duality. In general, the bar-cobar duality (and Koszul duality, where available) for generalised operads defined using Feynman categories involves the so called $\mathfrak{K}$-twists, see \cite{MR3636409}. However, the fact that the $\CGr$-coloured operad $\Rec$ has only binary operations and only quadratic relations, easily implies that one can get rid of the sign twists and obtain an honest duality between the categories of differential graded reconnectads and differential graded coreconnectads; in this way, the case of reconnectads much more similar to associative algebras (algebras over the Koszul self-dual operad of associative algebras) or operads \cite{VanDerLaan03} than to modular operads \cite{DSVV,MR4425832}. For that reason, we do not give a lot of detail on it: we indicate some key aspects of the theory, trusting that in the case of proofs that are nearly identical to those for the case of nonsymmetric operads the reader will be able to reconstruct the counterpart for reconnectads either on their own or with help of \cite{MR2954392}.  

We note that for each graphical collection $\calX$, the free reconnectad $\calN(\calX)$ has a standard weight grading for which the generators $\calX$ are in weight grading one. This grading is additive under compositions. Moreover, it is often the case that operads presented by generators and relations have homogeneous relations, and so there is an induced weight grading on the quotient. 

The main reason behind both the bar-cobar duality and the Koszul duality theory comes from searching for ``good'' models of algebraic objects. A model for a reconnectad $\calF$ is a differential graded reconnectad $\calM$ with a morphism $\calM\to\calF$ that induces an isomorphism on the homology. Under very mild assumptions, reconnectads have (unique up to isomorphism) minimal models. A \emph{minimal model} of a reconnectad is a differential graded reconnectad whose underlying reconnectad is free and whose differential is \emph{decomposable}: the differential of each generator is a combination of elements of weight grading strictly larger than one.

\subsubsection{Twisting morphisms}

Let us explain how the notion of a twisting morphism adapts to reconnectads; we outline the necessary statements, and all the proofs are \emph{mutatis mutandis} those of \cite[Sec.~6.4]{MR2954392}. 

Let $(\calG,\Delta)$ be a coreconnectad and $(\calF,\mu)$ be a reconnectad. Let us consider the graphical collection $\Hom(\calG,\calF)$ defined by
 \[
\Hom(\calG,\calF)(\Gamma)=\Hom_{\k}(\calG(\Gamma),\calF(\Gamma)). 
 \]
It has an obvious reconnectad structure defined as follows. To define the structure map
 \[
\Hom(\calG,\calF)(\Gamma^*_V)\otimes \Hom(\calG,\calF)(\Gamma_V)\to \Hom(\calG,\calF)(\Gamma),
 \]
one needs to be able to evaluate 
 \[
f\otimes g_1\otimes\cdots\otimes g_k\in \Hom(\calG,\calF)(\Gamma^*_V)\otimes\bigotimes_{\Gamma'\in\mathrm{Conn}(\Gamma_V)} \Hom(\calG,\calF)(\Gamma')
 \]
on an element $x\in\calG(\Gamma)$. For that, one applies first the coreconnectad decomposition map $\Delta_G^\Gamma$ to $x$, then the map $f\otimes g_1\otimes\cdots\otimes g_k$ to the respective tensor factors, and then the reconnectad composition map $\mu_G^\Gamma$ to the result. Moreover, if $\calG$ and $\calF$ are both differential graded, the differential $\partial$ of the Hom complex makes it into a differential graded reconnectad.

Let us explain how to associate to every reconnectad a pre-Lie algebra; a particular case of this construction is discussed in \cite[Rem.~3.1.3]{https://doi.org/10.1112/jlms.12596}. Given a reconnectad $\calF$, we define
 \[
\mathrm{Tot}(\calF):=\bigoplus_{\emptyset\ne V_\Gamma\subset\mathbb{N}}\calF(\Gamma),
 \]
set, for $\alpha\in\calF(\Gamma_1)$, $\beta\in\calF(\Gamma_2)$,
 \[
\alpha\star\beta=\sum_{\Gamma, T\colon \Gamma^*_T=\Gamma_1,\Gamma_T=\Gamma_2}\mu_T^\Gamma(\alpha\otimes\beta)
 \]
and then extend $\star$ to $\mathrm{Tot}(\calF)$ as a bilinear operation. Then the axioms of infinitesimal structure operations of a reconnectad listed in Proposition \ref{prop:infini} ensure that $(\mathrm{Tot}(\calF),\star)$ is a (right) pre-Lie algebra, that is
 \[
(a_1\star a_2)\star a_3-a_1\star(a_2\star a_3)=(-1)^{|a_2||a_3|}((a_1\star a_3)\star\ a_2-a_1\star(a_3\star a_2)).
 \] 

As it is in the case of operads, one can show that the space of equivariant maps 
 \[
\mathrm{Tot}(\Hom_{\Aut}(\calG,\calF)):=\bigoplus_{\emptyset\ne V_\Gamma\subset\mathbb{N}}\Hom_{\Aut(\Gamma)}(\calG(\Gamma),\calF(\Gamma))
 \]
is a pre-Lie subalgebra of $\mathrm{Tot}(\Hom(\calG,\calF)$. We shall refer to the Lie algebra obtained by anti-symmetrising the pre-Lie product $\star$ as the \emph{convolution dg Lie algebra of the coreconnectad $\calG$ and the reconnectad $\calF$. } The \emph{Maurer--Cartan elements} of that dg Lie algebra, that is degree $-1$ solutions to the equation
 \[
\partial(\alpha)+\alpha\star\alpha=0,
 \]
are of particular importance: they give nontrivial ways to twist the usual chain complex structure of the reconnected product $\calG\circ_{\rmR}\calF$ with a structure of a chain complex, denoted $\calG\circ^\alpha_{\rmR}\calF$. We denote by $\Tw(\calG,\calF)$ the set of all elements like that, and call them reconnectadic twisting morphisms.

\subsubsection{Bar-cobar duality}

In this section, we outline the main steps to construct the bar-cobar adjunction between reconnectads and coreconnectads. As above, we outline the necessary statements, and all the proofs are \emph{mutatis mutandis} those of \cite[Sec.~6.5]{MR2954392}. 

Let $\calF$ be a differential graded reconnectad $\calF$. Let us define two coderivations of the cofree coreconnectad $\calN^c(s\calF)$. The coderivation $d_1$ is the unique extension of the map
 \[
\calN^c(s\calF)\twoheadrightarrow s\calF\to s\calF,
 \]
where the first arrow is the obvious projection of graphical collections, and the second arrow corresponds to the differential of $\calF$. The coderivation $d_2$ is the unique extension of the map
 \[
\calN^c(s\calF)\to s\calF 
 \]
obtained, in each arity $\Gamma$, as the projection onto
 \[
\bigoplus_{T\in G_\Gamma}s\calF(\Gamma^*_T)\otimes s\calF(\Gamma_T)
 \]
for all $T\in G_\Gamma$, followed by the (de)suspended structure map 
 \[
\mu_T^\Gamma\colon \calF(\Gamma^*_T)\otimes \calF(\Gamma_T)\to \calF(\Gamma).
 \] 
It is easy to see that the property of $\calF$ to be a differential graded reconnectad can be compactly written as $d_1^2=0$, $d_1d_2+d_2d_1=0$, $d_2^2=0$, implying that we have $(d_1+d_2)^2=0$.

\begin{definition}[bar construction]
For a differential graded reconnectad $\calF$, the bar construction $\mathsf{B}(\calF)$ is the differential graded coreconnectad 
 \[
(\calN^c(s\calF),d_1+d_2).
 \]
\end{definition}

Let us note that in the case of operads, one defines the bar construction by building the cofree cooperad on the \emph{augmentation ideal}. 
In the universe of reconnectads, the assumption $\calF(\emptyset)=1_{\sfC}$ takes care of many problems: each graphical collection $\calF$, whether given a reconnectad structure or not, has the unit attached to it as $\calF(\emptyset)$, and the construction of a free reconnectad on a graphical collection in fact constructs the free reconnectad on the corresponding augmentation ideal. In addition, the number of vertices of the graph is, tautologically, a weight grading that may be used where an operadic analogue would need it as an extra condition.  

Dually, for a differential graded coreconnectad $\calG$, one may define two derivations of the free reconnectad $\calN(s^{-1}\calG)$. The derivation $d_1$ is the unique extension of the map
 \[
s^{-1}\calG\to s^{-1}\calG\hookrightarrow\calN(s^{-1}\calG),
 \]
where the first arrow corresponds to the differential of $\calG$ and the second arrow is the obvious inclusion of graphical collections. The derivation $d_2$ is the unique extension of the map
 \[
s^{-1}\calG\to \calN(s^{-1}\calG) 
 \]
obtained, in each arity $\Gamma$, as the (de)suspended structure maps 
 \[
\Delta_T^\Gamma\colon  \calG(\Gamma)\to \calG(\Gamma^*_T)\otimes \calG(\Gamma_T)
 \]
for all $T\in G_\Gamma$, followed by the inclusion of
 \[
\bigoplus_{T\in G_\Gamma}s\calG(\Gamma^*_T)\otimes s\calG(\Gamma_T)
 \]
into the free reconnectad. It is easy to see that the property of $\calG$ to be a differential graded coreconnectad can be compactly written as $d_1^2=0$, $d_1d_2+d_2d_1=0$, $d_2^2=0$, implying that we have $(d_1+d_2)^2=0$.

\begin{definition}[cobar construction]
For a differential graded coreconnectad $\calG$, the cobar construction $\Omega(\calG)$ is the differential graded reconnectad 
 \[
(\calN(s^{-1}\calG),d_1+d_2).
 \]
\end{definition}

The following result is completely analogous to \cite[Th.~6.5.10]{MR2954392}.

\begin{proposition}
The bar and the cobar construction form an adjoint pair; moreover, we have for every dg reconnectad $\calF$ and every dg coreconnectad $\calG$
 \[
\mathsf{Hom}_{\mathrm{dg\ rec}}(\Omega(\calG), \calF)\cong\Tw(\calG,\calF) \cong \mathsf{Hom}_{\mathrm{dg\ corec}}(\calG, \mathsf{B}(\calF)).
 \]
\end{proposition} 

For every reconnectad $\calF$, there is the ``universal twisting morphism'' $$\pi\colon\mathsf{B}(\calF)\to\calF$$ obtained as projection of $\calN^c(s\calF)\to s\calF $ followed by desuspension. It is easy to check that the twisted reconnected product $\mathsf{B}(\calF)\circ^\pi_{\rmR}\calF$ is an acyclic complex.

\begin{definition}[Koszul twisting morphism]
For a coreconnectad $\calG$ and a reconnectad $\calF$, a twisting morphism $\alpha\colon\calG\to\calF$ is said to be a \emph{Koszul morphism} if the twisted reconnected product $\calG\circ^\alpha_{\rmR}\calF$ is acyclic. 
\end{definition}

The following result is completely analogous to \cite[Th.~6.6.2]{MR2954392}.

\begin{proposition}
A twisting morphism $\alpha$ is a Koszul morphism if and only if either of the maps $f_\alpha\colon\calG\to\mathsf{B}(\calF)$ and $g_\alpha\colon\Omega(\calG)\to\calF$ is a quasi-isomorphism.
\end{proposition}

This result immediately implies the following corollary.

\begin{corollary}
For every reconnectad $\calF$, the cobar-bar construction $\Omega(\mathsf{B}(\calF))$ is quasi-isomorphic to $\calF$. 
\end{corollary}

\subsubsection{Koszul duality}

In this section, we summarise the main aspects of the Koszul duality theory for reconnectads. As above, we outline the necessary statements, and all the proofs are \emph{mutatis mutandis} those of \cite[Chap.~7]{MR2954392}. 

As in the case of algebras and operads, the most manageable situation arises in the case of \emph{quadratic} reconnectads, for which all relations are of weight two. 

\begin{definition}[Koszul dual coreconnectad]
Let $\calF$ be a reconnectad with generators $\calX$ and quadratic relations $\calR$, which we shall refer to as \emph{quadratic data}. To such a reconnectad, one may associate its \emph{Koszul dual coreconnectad} $\calF^{\text{!`}}$, defined as the coreconnectad with cogenerators $s\calX$ and corelations $s^2\calR$. 
\end{definition}

We shall also consider the Koszul dual reconnectad obtained, like in the case of operads, by dualising and (de)suspending, as follows.

\begin{definition}[(de)suspension reconnectad] For $X=\k s^{-1}$, the commutative reconnectad $\grCom_X$ is called the \emph{suspension reconnectad}, and is denoted $\mathsf{S}$. In the same vein, for $X=\k s$, the commutative reconnectad $\grCom_X$ is called the \emph{desuspension reconnectad}, and is denoted $\mathsf{S}^{-1}$. 

We use these reconnectads to define suspensions and desuspensions for arbitrary graphical collections by the formulas
\begin{gather*}
\mathsf{S}\calF:=\mathsf{S}\underset{H}{\otimes}\calF,\\
\mathsf{S}^{-1}\calF:=\mathsf{S}^{-1}\underset{H}{\otimes}\calF.
\end{gather*}
Note that if $\calF$ is a reconnectad, then its suspension and desuspension are both reconnectads.
\end{definition} 

The reader may think that the chosen terminology for (de)suspension is somewhat counterintuitive. It is chosen in such way to match the operadic suspension under the equivalence of Proposition \ref{prop:NS}. 

\begin{definition}[Koszul dual reconnectad]
For a reconnectad $\calF$ presented by generators $\calX$ and relations $\calR$, we define the Koszul dual reconnectad $\calF^!$ by the formula
 \[
\calF^!=\mathsf{S}^{-1}(\calF^{\text{!`}})^*.
 \]
\end{definition} 

The following result is analogous to \cite[Prop.~7.2.4]{MR2954392}.

\begin{proposition}
Suppose that all components of the reconnectad $\calF$ are finite dimensional. The reconnectad $\calF^!$ is presented by the quadratic data of generators $s^{-1}\mathsf{S}^{-1}\calX^*$ and relations $\calR^\bot$, where $\bot$ refers to the annihilator under the natural pairing. 
\end{proposition}

Let $\calF$ be a quadratic reconnectad. We have two natural morphisms: inclusion of dg coreconnectads $\calF^{\text{!`}} \hookrightarrow \mathsf{B}(\calF)$ and surjection of reconnectads $\Omega(\calF^{\text{!`}}) \twoheadrightarrow \calF$. Moreover, they both correspond to the same twisting morphism between $\calF^{\text{!`}}$ and $\calF$ which projects the former onto cogenerators, desuspends, and includes the result as the space of generators. We say that a reconnectad is \emph{Koszul} if that twisting morphism is a Koszul morphism.

Let us give the first nontrivial example of a Koszul reconnectad. 

\begin{proposition}\label{prop:cobar-associahedra}
The reconnectad $\grCom$ is Koszul.
\end{proposition}

\begin{proof}
Let us begin with noting that the Koszul dual coreconnectad $\grCom^{\text{!`}} \cong \mathsf{S}^*$ is isomorphic to the linear dual of suspension coreconnectad. Denote by $\beta_{\Gamma} \in S^{*}(\Gamma)$ the basis element of degree $|V_\Gamma|$ which is the linear dual of the basis element $\bigotimes_{v\in V_\Gamma}s^{-1}$. Note that the infinitesimal decomposition $\Delta^\Gamma_T$ on $\grCom^{\text{!`}}$ is given by the formula
 \[
\Delta^\Gamma_T(\beta_{\Gamma})=\mathrm{sgn}(\sigma_T^{\Gamma})\beta_{\Gamma^*_T}\otimes b_{\Gamma_T},
 \] 
with the sign $\mathrm{sgn}(\sigma_T^{\Gamma})$ is the Koszul sign coming from the permutation separating the subset $T$ in the tensor product giving $\beta_{\Gamma}$. The coreconnectad $\grCom^{\text{!`}}$ has zero differential, so the differential of the cobar construction $\Omega(\grCom^{\text{!`}})$ is given by the differential $d_2$ on the free reconnectad on generators $\epsilon_{\Gamma}=s^{-1}\otimes \beta_{\Gamma}$ of degree $|V_\Gamma|-1$. That differential acts on generators by the rule
 \[
d_2(\epsilon_{\Gamma})=\sum_{T\in G_\Gamma}(-1)^{|V_\Gamma|}\mathrm{sgn}(\sigma_T^{\Gamma})\epsilon_{\Gamma^*_T}\otimes \epsilon_{\Gamma_T}.
 \]
A particular case of this formula is displayed in Figure~\ref{fig:differential}. 
\begin{figure}[ht]
    \centering
    \def\svgwidth{\columnwidth}
    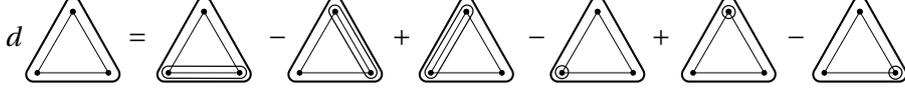
   \caption{The differential of the generator $\epsilon_{C_3}$}
   \label{fig:differential} 
\end{figure}

The remarkable property of graph associahedra stating that each facet of $\calP\Gamma$ is a product of two smaller graph associahedra \cite[Th.~2.9]{MR2239078} implies that the graphical collection of all cellular complexes $\{C^{\mathrm{cell}}_\bullet(\calP\Gamma)\}$ is isomorphic to the cobar construction
$\Omega(\grCom^{\text{!`}})(\Gamma)$ that we consider; the isomorphism sends the cell corresponding to the interior of $\calP\Gamma$ to the generator $\epsilon_{\Gamma}$. Since each graph associahedron is contractible, we have
 \[
H_{\bullet}(\Omega(\grCom^{\text{!`}})(\Gamma),\k) \cong H_{\bullet}(\calP\Gamma,\k)=
\begin{cases}
 \k \text{, for } \bullet=0
     \\
   \,  0 \text{, for } \bullet \neq 0.
     \end{cases}
 \]
In other words, the homology of the cobar construction is concentrated in degree $0$, and therefore the canonical projection 
 \[
\Omega(\grCom^{\text{!`}})\twoheadrightarrow\grCom
 \]
is a quasi-isomorphism.
\end{proof}

This result is very similar to that in \cite{https://doi.org/10.1112/jlms.12596}. However, the context in which the latter proof is given uses graphs with total orders on vertices, and hence is closer to shuffle reconnectads discussed in Section \ref{sec:shuffle} below. In any case, both of these proofs are, via the results discussed in Section \ref{sec:part-types}, common generalisations of two known results. First, the usual proof of Koszulness of the nonsymmetric associative operad $\mathsf{As}$ \cite{MR2954392}; that proof views its minimal model, the $A_\infty$ operad, as the collection of cellular complexes of Stasheff associahedra \cite{MR0158400}. Second, the proof of Koszulness of the associative permutad $\mathsf{permAs}$ \cite{MR2995045} views its minimal model as the collection of cellular complexes of permutahedra.

\subsubsection{Distributive laws for reconnectads}

It is possible to adapt the formalism of distributive laws \cite{MR1393517} to the case of reconnectads.

\begin{definition}[distributive law]
Let $\calF$ and $\calG$ be two reconnectads. We say that a morphism of graphical collections
\[
    \Lambda\colon \calG \circ_{\rmR} \calF \rightarrow \calF \circ_{\rmR} \calG
\]
defines a \emph{distributive law} between $\calF$ and $\calG$ if the composite 
 \[
\calF\circ_{\rmR}\calG\circ_{\rmR}\calF\circ_{\rmR}\calG\stackrel{\mathrm{id}\circ_{\rmR}\Lambda\circ_{\rmR}\mathrm{id} }{\longrightarrow}  
\calF\circ_{\rmR}\calF\circ_{\rmR}\calG\circ_{\rmR}\calG
\stackrel{\mu_{\calF}\circ_{\rmR}\mu_{\calG}}{\longrightarrow}\calF\circ_{\rmR}\calG
 \]

defines a reconnectad structure on the graphical collection $\calF\circ_{\rmR}\calG$.
\end{definition}

Let us define the infinitesimal product of two graphical collections by the formula
 \[
(\calF \circ'_{\rmR} \calG)(\Gamma) = \bigoplus_{T\in G_\Gamma} \calF(\Gamma^*_{V}) \otimes \calG(\Gamma_{V}).
\] 

Given two reconnectads $\calF$ and $\calG$ presented by generators $\calX$ and $\calY$ and relations $\calR$ and $\calS$ respectively, suppose that we are given a morphism of graphical collections $\lambda\colon \calY \circ'_{\rmR} \calX \rightarrow \calX\circ'_{\rmR}\calY$. We may now define a reconnectad $\calF\vee_{\lambda}\calG$ as the reconnectad with generators $\calX\oplus\calY$ and with relations
 \[
\calR\oplus\calS\oplus\langle z-\lambda(z)\colon z\in \calY \circ'_{\rmR} \calX\rangle.
 \]

The following result is proved analogously to \cite[Prop.~8.6.4]{MR2954392}.

\begin{proposition}
For every choice of $\lambda\colon \calY \circ'_{\rmR} \calX \rightarrow \calX\circ'_{\rmR}\calY$, there is a surjective map of graphical collections $\calF\circ_{\rmR}\calG\to \calF\vee_{\lambda}\calG$. If that map is an isomorphism, the composite
\[
\calG\circ_{\rmR} \calF \rightarrow (\calF\vee_{\lambda}\calG)\circ_{\rmR} (\calF\vee_{\lambda}\calG)\to \calF\vee_{\lambda}\calG \cong\calF\circ_{\rmR} \calG
\]
is a distributive law between the reconnectads $\calF$ and $\calG$.
\end{proposition}

Let us give an example of a distributive law. For that, we shall define a reconnectad analogue of the Gerstenhaber operad.

\begin{definition}[Gerstenhaber reconnectad]
The reconnectad $\grCom_{H_\bullet(S^1)}$ is called the \emph{Gerstenhaber reconnectad}, and is denoted $\grGerst$.
\end{definition}

Proposition \ref{prop:relgrcom} easily implies the following result.

\begin{proposition}\label{prop:relgerst}
The Gerstenhaber reconnectad is generated by by the graphical collection $\calX$ supported at $P_1$ for which $\calX(P_1)=H_\bullet(S^1,\k)$; if we denote by $m,b\in \grGerst(\mathsf{P}_1)$ the basis elements of homological degrees $0$ and $1$ respectively corresponding to $[\mathrm{pt}]$ and $[S^1]$ respectively, a complete system of defining relations is
\begin{align}
m\circ_{1}^{P_2} m - m\circ_{2}^{P_2} m &= 0\\
b\circ_{1}^{P_2} b + b\circ_{2}^{P_2} b &= 0\\
m\circ_{1}^{P_2} b - b\circ_{2}^{P_2} m &= 0\\
b\circ_{1}^{P_2} m - m\circ_{2}^{P_2} b &= 0
\end{align}
\end{proposition}

\begin{proof}
This follows from Proposition~\ref{prop:grEnd-monoidal} and the obvious isomorphism
 \[
H_{\bullet}(\grCom_X,\k)\cong \grCom_{H_{\bullet}(X,\k)}.
 \]
\end{proof}

Graphically the generators of $\grGerst$ and the relations between them are displayed in Figure~\ref{fig:gerstrel}. 
\begin{figure}[ht]
    \centering
    \def\svgwidth{0.7\columnwidth}
    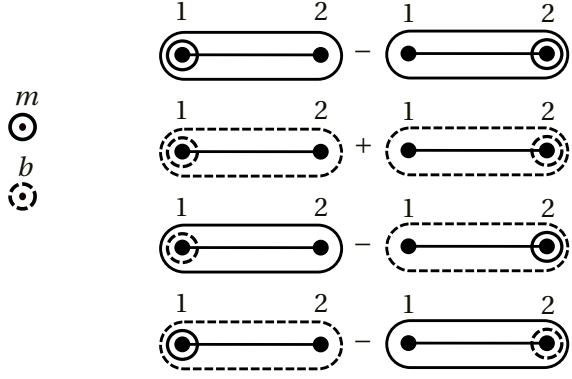
    \caption{Generators and relations for $\grGerst$}
    \label{fig:gerstrel}
\end{figure}

Let us remark that we do not have to include explicitly relations obtained by action of graph automorphisms: the action of the transposition $\sigma=(1,2)$ exchanges the third relation with the fourth one. Thus, the first three relations of Proposition \ref{prop:relgerst} are sufficient to define the reconnectad $\grGerst$. This is related to Proposition \ref{prop:NS}: restricting to path graphs more or less recovers the noncommutative Gerstenhaber operad \cite{MR4072173}, but with the additional symmetries allowing to have fewer elements in the minimal set of relations. 

\begin{proposition}\label{prop:gerst-distr}
The reconnectad $\grGerst$ is obtained from $\grCom$ and $\mathsf{S}^{-1}$ by a distributive law. In particular, the underlying graphical collection of that reconnectad permutad is isomorphic to $\grCom\circ_{\rmR}\mathsf{S}^{-1}$.
\end{proposition}

\begin{proof}
Note that the reconnectad $\grGerst$ is obtained from $\grCom$ and $\mathsf{S}^{-1}$ by the rewriting rule 
 \[
\lambda\colon b\circ_1m \mapsto m\circ_2b,
 \]
leading to a surjection of graphical collections
 \[
\grCom \circ_{\rmR} \mathsf{S}^{-1}\twoheadrightarrow \grCom \vee_{\lambda} \mathsf{S}^{-1} \cong \grGerst.
 \]
This surjection is an isomorphism by a dimension argument. Indeed, we have 
 \[
\dim ((\grCom \circ_{\rmR} \mathsf{S}^{-1})(\Gamma))=\sum_{V\subset V_\Gamma} \dim (\grCom(\Gamma^*_{V})\otimes\mathsf{S}^{-1}(\Gamma_V))=
\sum_{V\subset V_\Gamma} 1=2^{|V_\Gamma|},
 \]
which coincides with $\dim \grGerst(\Gamma)=\dim\grCom_{H_\bullet(S^1)}(\Gamma)$. 
\end{proof}

It is well known that an operad obtained from two Koszul operads by a distributive law is Koszul. An analogous result holds for reconnectads: if $\calF$ and $\calG$ be two Koszul reconnectads, and a morphism of graphical collections
\[
    \Lambda\colon \calG \circ_{\rmR} \calF \rightarrow \calF \circ_{\rmR} \calG
\]
defines a distributive law between $\calF$ and $\calG$, then the corresponding reconnectad $\calF\vee_{\lambda}\calG$ is Koszul. Moreover, if the restriction of the canonical map $\calF\circ_{\rmR}\calG\twoheadrightarrow \calF\vee_{\lambda}\calG$ to the elements of weight three is injective, then $\Lambda$ defines a distributive law. The proofs of these results repeat \emph{mutatis mutandis} the corresponding proofs of \cite[Sec.~8.6]{MR2954392}.

\begin{corollary}
The Gerstenhaber reconnectad is Koszul.
\end{corollary}

\subsection{Gr\"obner bases for reconnectads}\label{sec:shuffle}

In this section, we develop the core of the theory of Gr\"obner bases for reconnectads. As above, we outline the necessary statements; most of the proofs are \emph{mutatis mutandis} those of \cite[Sec.~8.2]{MR2954392}. The main idea is to create a framework in which there is sufficiently many reconnectads with monomial relations. This is not the case so far: due to the presence of graph automorphisms, vanishing of a monomial forces ``unexpected'' vanishing of other monomials. To avoid that, we shall consider a version of graphical collections where we use connected graphs equipped with the additional data of a total order on the set of vertices: this eliminates automorphisms from the picture, since an automorphism is required to preserve all structures, including the order.  
We shall call those nonsymmetric graphical collections. The shuffle reconnected product of two nonsymmetric graphical collections is defined by the formula
  \[
(\calF\circ_{\Sha}\calG)(\Gamma):=\calF(\Gamma)\oplus\bigoplus_{\emptyset\neq V\subset V_\Gamma}\calF(\Gamma^*_V)\otimes
\bigotimes_{\substack{\mathrm{Conn}(\Gamma_V)=\{\Gamma_1,\ldots,\Gamma_s\}\\ \min (V_{\Gamma_1})<\cdots<\min (V_{\Gamma_s})}} \calG(\Gamma_1)\otimes\cdots\calG(\Gamma_s).
 \]
This operation can be easily shown to give the category of nonsymmetric graphical collections a structure of a monoidal category. By definition, a \emph{shuffle reconnectad} is a monoid in that monoidal category. 

The importance of the shuffle reconnected product comes from its compatibility with the forgetful functor $\Gamma \mapsto \Gamma^{\mathrm{f}}$ from the groupoid of connected graphs with a total order on the set of vertices to the groupoid of connected graphs; that functor literally forgets the total order of the set of vertices of each graph. This functor defines a forgetful functor from graphical collections to nonsymmetric graphical collections given by 
 \[
\calF^{\mathrm{f}}(\Gamma)=\calF(\Gamma^{\mathrm{f}}).
 \]  
Moreover, there is an obvious analogue of the shuffle nested set monad, and one can use it to define the free shuffle reconnectad $\calN_{\Sha}(\calX)$ generated by a nonsymmetric graphical collection $\calX$. Using that notion, one recovers all other aspects of the reconnectad theory in the shuffle case: it is possible to use presentations by generators and relations, the (co)bar construction, Koszul duality etc.

\begin{proposition}\label{prop:forget-monoidal}
The forgetful functor from nonsymmetric graphical collections to graphical collections is monoidal: we have
 \[
(\calF\circ_{\rmR}\calG)^{\mathrm{f}}\cong \calF^{\mathrm{f}}\circ_{\Sha}\calG^{\mathrm{f}}.
 \]
\end{proposition}

\begin{proof}
The symmetric group acts freely on the unordered tensor product over the connected components; the shuffle product corresponds to one of the ways 
to choose a representative of each orbit (which, in the case of shuffle graphical collections, is canonical).
\end{proof}

An immediate consequence of Proposition \ref{prop:forget-monoidal} is the following result. 

\begin{corollary}\label{cor:forget}\leavevmode
\begin{enumerate}
\item For every reconnectad $\calF$, the nonsymmetric collection $\calF^{\mathrm{f}}$ has a natural structure of a shuffle reconnectad obtained from that of $\calF$.
\item For every graphical collection $\calX$, we have an isomorphism of shuffle reconnectads
 \[
\calN(\calX)^{\mathrm{f}}\cong\calN_{\Sha}(\calX^{\mathrm{f}}). 
 \]
\item For every graphical collection $\calX$ and every graphical subcollection $\calR\subset\calN(\calX)$, we have an isomorphism of shuffle reconnectads
 \[
(\calN(\calX)/\langle\calR\rangle)^{\mathrm{f}}\cong \calN_{\Sha}(\calX^{\mathrm{f}})/\langle\calR^{\mathrm{f}}\rangle).
 \]
\item For every reconnectad $\calF$, we have an isomorphism of dg shuffle coreconnectads
 \[
\mathsf{B}(\calF)^{\mathrm{f}}\cong\mathsf{B}_{\Sha}(\calF^{\mathrm{f}}).
 \] 
\item A reconnectad $\calF$ presented by generators and quadratic relations is Koszul if and only if the associated shuffle reconnectad $\calF^{\mathrm{f}}$ is Koszul.
\end{enumerate}
\end{corollary}

Suppose that the given nonsymmetric graphical collection $\calX$ with values in vector spaces is the linearisation of a nonsymmetric graphical collection $\calB$ with values in sets; in other words, $\calB(\Gamma)$ is a functorial choice of a basis in $\calX(\Gamma)$. Then the free shuffle reconnectad $\calN_{\Sha}(\calX)$ is the linearisation of the free shuffle reconnectad $\calN_{\Sha}(\calB)$, and we can talk about monomials $\calN_{\Sha}(\calX)$. We say that a monomial $m$ is \emph{divisible} by another monomial $m'$ if it can be obtained from $m'$ by iteration of structure operations. A collection of total well orders of all sets $\calN_{\Sha}(\calB)(\Gamma)$ is said to be a \emph{monomial ordering} if each structure operation of $\calN_{\Sha}(\calB)$ is strictly increasing in each argument. 

Let us give an example of a monomial ordering, which is a particular case of a more general ordering introduced by Coron \cite{CoronLattices}. We shall consider the free shuffle reconnectad with one generator in each ``arity'', so that the basis of each component $\calN_{\Sha}(\calX)(\Gamma)$ is indexed by $N^+(\Gamma)$. 

As a preparation, we shall define an order on subsets of $V_{\Gamma}$. For two such subsets $V=\{v_1,\ldots,v_r\}$ with $v_1<\cdots<v_r$ and $W=\{w_1,\ldots,w_s\}$ with $w_1<\cdots<w_s$, we shall say that $V$ precedes $W$ and write $V\prec W$ if either the sequence $(v_1,\ldots,v_r)$ coincides with an initial segment of the sequence $(w_1,\ldots,w_s)$ or the sequence $(v_1,\ldots,v_r)$ is lexicographically greater than the sequence $(w_1,\ldots,w_s)$. Note that this order extends the existing order on $2^{V_\Gamma}$: $V\subset W$ implies that $V\prec W$. 

When we pass to induced subgraphs $\Gamma_V$ and reconnected complements $\Gamma^*_V$, there are two possible way to proceed. One can either induce the total orders on the sets of vertices of these graphs from the total order on $V_\Gamma$ and then define the relation $V\prec W$, or directly induce the relation $V\prec W$ from $2^{V_\Gamma}$. By a direct inspection, these two recipes give the same result. 

We define the \emph{lexicographic} ordering $\triangleleft$ allowing to compare two nested sets of the same cardinality (not necessarily in $N^+(\Gamma)$) as follows. For two nested sets $\tau_1,\tau_2$, we say that $\tau_1$ is \emph{lexicographically smaller} than $\tau_2$, and write $\tau_1 \triangleleft \tau_2$, if $\bigcup_{T\in \tau_1} T \prec \bigcup_{T\in \tau_2} T$ or if $\bigcup_{T\in \tau_1} T = \bigcup_{T\in \tau_2} T$ and 
 \[
\bigcup_{T\in \tau_1} T\setminus \{\max\nolimits_{\prec} \tau_1\} \prec \bigcup_{T\in \tau_2} T\setminus \{\max\nolimits_{\prec} \tau_2\}.
 \]
 
\begin{proposition}[{\cite[Sec.~5.3]{CoronLattices}}]
The ordering of monomials $\calN_{\Sha}(\calX)(\Gamma)$ that compares the corresponding elements of $N^+(\Gamma)$ using the ordering $\triangleleft$  is compatible with reconnectad structure.
\end{proposition}

To define orderings of monomials of arbitrary free shuffle reconnectads, one has to blend the lexicographic ordering with an ordering of words in basis elements, mimicking the strategy of \cite{MR4114993}. We leave the precise definition as an exercice to the reader who wishes to consider reconnectads where such a definition is necessary. 

Let us fix a certain monomial ordering; this allows us to talk about leading terms of elements of the free shuffle reconnectad. Suppose that $\calI$ is an ideal of the free shuffle reconnectad $\calN_{\Sha}(\calX)$ We say that a collection of subsets $\calG(\Gamma)\subset\calI(\Gamma)$ is a \emph{Gr\"obner basis} of the ideal $\calI$, if every leading monomial of each element of $\calI$ is divisible by a leading term of one of the elements of $\calG$. 

Let us define \emph{normal monomials} with respect to a collection of subsets $\calS(\Gamma)\subset\calI(\Gamma)$ as monomials that are not divisible by leading monomials of elements of $\calS$. The following simple but important observation is proved analogously to all other known instances of Gr\"obner bases \cite{MR3642294}.

\begin{proposition}\label{prop:normal} Let $\calX$ be a nonsymmetric graphical collection, and suppose that a collection of subsets $\calG(\Gamma)\subset\calI(\Gamma)$ generates $\calI$ as an ideal in the free shuffle reconnectad $\calN_{\Sha}(\calX)$.
\begin{enumerate}
\item Cosets of normal monomials with respect to $\calG$ form a spanning set in the quotient reconnectad $\calN_{\Sha}(\calX)/\calI$.
\item The collection $\calG$ is a Gr\"obner basis of $\calI$ if and only if the cosets of normal monomials with respect to $\calG$ form a basis in the quotient reconnectad $\calN_{\Sha}(\calX)/\calI$. 
\end{enumerate}
\end{proposition}

The following key result is proved analogously to the corresponding result for shuffle operads \cite{MR3084563}.

\begin{proposition}\label{prop:shuffleKoszul}\leavevmode
\begin{enumerate}
\item A shuffle reconnectad with quadratic monomial relations is Koszul.
\item A shuffle reconnectad with a quadratic Gr\"obner basis of relations (for some monomial ordering) is Koszul.
\end{enumerate}
\end{proposition}

This result is absolutely fundamental, since Corollary \ref{cor:forget} implies that to prove that a reconnectad $\calF$ is Koszul, it is sufficient to establish that the associated shuffle reconnectad $\calF^{\mathrm{f}}$ is Koszul, and we now know that for that it is enough to exhibit a quadratic Gr\"obner basis.

\begin{example}\label{ex:grCom}
Let us consider the free shuffle reconnectad in vector spaces generated by the graphical collection $\calX$ supported at $P_1$ with $\calX(P_1)=\k m$. For the restriction of the lexicographic ordering to this case, the ideal of relations of the shuffle reconnectad $\grCom^{\mathrm{f}}$ admits a quadratic Grobner basis with leading term $m\circ_{1}^{\mathsf{P}_2}m$. Indeed, let us define, for each graph $\Gamma$, an element $m_\Gamma \in \grCom^{\mathrm{f}}(\Gamma)$ by the inductive rule 
 \[
m_\Gamma = m_{\Gamma^*_{V_\Gamma\setminus\{\max(V_\Gamma)\}}}\circ_{\max(V_\Gamma)}^{\Gamma} m_{\max(V_\Gamma)},
 \]
which in plain words means that we disassemble each graph starting from its maximal vertex. The element $m_\Gamma$ forms a basis of the one-dimensional component $\grCom^{\mathrm{f}}(\Gamma)$, and it is the only normal form with respect 
to the leading term $m\circ_{1}^{\mathsf{P}_2}m$. Thus, Propositions \ref{prop:normal} and \ref{prop:shuffleKoszul} together with Corollary \ref{cor:forget} imply that the reconnectad $\grCom$ is Koszul, giving an alternative short proof of the result of Proposition \ref{prop:cobar-associahedra}.
\end{example}

Let us also record the following useful observation proved in the same way as the analogous result for associative algebras \cite[Th.~4.1]{MR2177131} and shuffle operads \cite[Th.~5.5]{MR2574993}.

\begin{proposition}\label{prop:dualPBW}
A shuffle reconnectad $\calF$ has a quadratic Gr\"obner basis for a certain monomial ordering if and only is the Koszul dual shuffle reconnectad $\calF^!$ has a quadratic Gr\"obner basis for the opposite monomial ordering.  
\end{proposition}

\section{The wonderful Koszul pair}\label{sec:wonderful}

In this section, we discuss various aspects of geometry and topology of one of the most interesting reconnectads we are aware of: the complex wonderful reconnectad $\calW_{\mathbb{C}}$. Our main result exhibits what we call ``a wonderful Koszul pair'', a reconnectad analogue of the operads of hypercommutative algebras and of gravity algebras introduced and studied by Getzler \cite{MR1284793,MR1363058}.

\subsection{The gravity reconnectad}
In this section, we imitate the construction of the gravity operad due to Getzler \cite{MR1284793}. 

Note that the reconnectad $\grGerst$ has a degree $1$ derivation $\mathrm{d}$ with $\mathrm{d}^2=0$ such that $\mathrm{d}(m)=b$. This follows from the fact that $\grCom_{S^1}$ has the diagonal circle action, and hence on the homology there is an induced infinitesimal action of $H_\bullet(S^1)$; the derivation $\mathrm{d}$ is the action of the class $[S^1]\in H_\bullet(S^1)$. 

\begin{proposition}\label{prop:acyclic}
The cochain complex $(\grGerst, \mathrm{d})$ is acyclic. 
\end{proposition}

\begin{proof}
Let us prove the result for the dual chain complex, that is the graphical collection $\grCom(H^{\bullet}(S^1))\cong H^\bullet(\grCom_{S^1})$. 
Since cohomology is contravariant, the natural structure on this collection is that of a coreconnectad. One aspect that is advantageous in this viewpoint is that the cohomology has an algebra structure, and, in this particular situation, a very simple one: $H^\bullet(\grCom_{S^1})(\Gamma)$ is the exterior algebra on generators $\omega_v$, $v\in V_\Gamma$. The infinitesimal coreconnectad structure maps 
 \[
H^\bullet(\grCom_{S^1})(\Gamma)\to H^\bullet(\grCom_{S^1})(\Gamma^*_T)\otimes H^\bullet(\grCom_{S^1})(\Gamma_T)
 \] 
are compatible with the algebra structures: they are the algebra homomorphisms defined by 
\[
 \omega_v \mapsto \begin{cases}
    \omega_v\otimes1, v\not\in T
    \\
    1\otimes\omega_v, v \in T
    \end{cases}.
\]
Moreover, the dual of the cochain differential $\mathrm{d}$ corresponds to the derivation $\partial$ of the algebra $H^\bullet(\grCom_{S^1})(\Gamma)$ defined by setting $\partial(\omega_v)=1$ for all $v\in V_\Gamma$. We would like to show that the differential $\partial$ makes the dual Gerstenhaber copermutad an acyclic chain complex. We shall now exhibit a contracting homotopy for that complex. We first note that for every graph $\Gamma$ and every vertex $v \in V_\Gamma$, the map 
 \[
H_v\colon H^\bullet(\grCom_{S^1})(\Gamma)\to H^\bullet(\grCom_{S^1})(\Gamma)
 \]
defined as multiplication by $\omega_v$ on the left is a contracting homotopy for $\partial$ in the category of chain complexes:
 \[
(\partial H_v +H_v \partial)(x)=\partial(\omega_vx)+\omega_v\partial x=(x-\omega_v\partial x)+\omega_v\partial x=x.
 \]
However, we would like to construct a contracting homotopy in the category of graphical collections. For that, we average the homotopies over all vertices, defining
 \[
H:=\frac{1}{|V_\Gamma|}\sum_{v \in V_\Gamma} H_v.
 \] 
This map is still a contracting homotopy but now it commutes with the action of $\mathrm{Aut}(\Gamma)$, so it is a contracting homotopy in the category of graphical collections.
\end{proof}

The kernel of a derivation of a reconnectad is itself a reconnectad, so we may give the following definition.

\begin{definition}[gravity reconnectad]
The \emph{gravity reconnectad}, denoted $\grGrav$, is the kernel 
 \[
\ker \mathrm{d} \subset \grGerst.
 \]
\end{definition}

Let us record a useful dimension formula.

\begin{proposition}\label{prop:gravdim}
For every nonempty graph $\Gamma$, we have
\[
\dim \grGrav(\Gamma) = 2^{|V_\Gamma|-1}.
\]
\end{proposition}

\begin{proof}
This is an immediate consequence of Proposition~\ref{prop:acyclic} and the formula $\dim \grGerst(\Gamma) = 2^{|V_\Gamma|}$.
\end{proof}

The main result of this section is a presentation of the reconnectad $\grGrav$ by generators and relations. 

\begin{theorem}\label{th:gravpresent}
The reconnectad $\grGrav$ is generated by the elements $\lambda_{\Gamma} \in \grGrav(\Gamma)$ of homological degree $1$ which are $\mathrm{Aut}(\Gamma)$-invariant and satisfy the relations
\begin{gather}
\sum_{v \in T} \lambda_{\Gamma^*_v} \circ_v^{\Gamma} \lambda_{\Gamma_v} = \lambda_{\Gamma^*_T}\circ_T^{\Gamma} \lambda_{\Gamma_T},\label{eq:grav1} \\
    \sum_{v \in V_\Gamma}  \lambda_{\Gamma^*_v} \circ_v^{\Gamma} \lambda_{\Gamma_v} = 0 .\label{eq:grav2}
\end{gather}
In these relations, $\Gamma$ is an arbitrary graph, and $V_\Gamma\ne T\in G_\Gamma$ . The reconnectad $\grGrav$ is Koszul.
\end{theorem}

A particular case of these relations for the cycle $C_3$ is given in Figure~\ref{fig:gravrel}.
\begin{figure}[ht]
    \centering
    \def\svgwidth{0.7\columnwidth}
    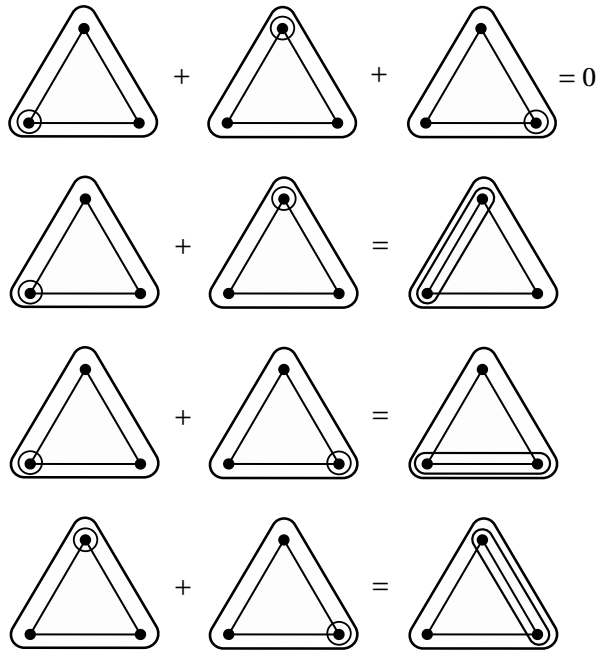
    \caption{Example of relations of the gravity reconnectad}
    \label{fig:gravrel}
\end{figure}

\begin{proof}
We note that, in order to have a nontrivial relation, we have to assume that $T\in G_\Gamma$ has cardinality at least two. Let us begin with indicating the elements $\lambda_{\Gamma}$ of $\grGrav(\Gamma)$ which we shall use as generators. Since the cochain complex $(\grGerst, \mathrm{d})$ is acyclic, we have 
 \[
\grGrav_1=\mathrm{d}\grGerst_0.
 \] 
The space of elements of homological degree zero in $\grGerst(\Gamma)$ is of dimension one, spanned by the basis element $m_\Gamma$ of $\grCom(\Gamma)$, and therefore the space of elements of homological degree $1$ in the reconnectad $\grGrav$ has a basis consisting of elements $\lambda_{\Gamma}=\mathrm{d}(m_{\Gamma})$. In terms of the generators of $\grGerst$, these elements can be computed as
 \[
\mathrm{d}(m_{\Gamma})=\partial^{*}(1^*_{\Gamma})=\sum_{v\in V_\Gamma} \omega_v^*=\sum_{v\in V_\Gamma} m_{\Gamma^*_v} \circ_v^\Gamma b_{\Gamma_v}.
 \]
Let us demonstrate that the set of elements $\lambda_{\Gamma}$ described above generates the reconnectad $\grGrav$. For that, we shall take a graph $\Gamma$ and choose a total order on $V_\Gamma$. Let us define, analogously to Example \ref{ex:grCom}, an element $b_\Gamma \in \grGerst(\Gamma)$ by the inductive rule 
 \[
b_\Gamma = b_{\Gamma^*_{V_\Gamma\setminus\{\max(V_\Gamma)\}}}\circ_{\max(V_\Gamma)}^{\Gamma} b_{\max(V_\Gamma)}.
 \]
The element $b_\Gamma$ forms a basis of the one-dimensional component $\mathsf{S}^{-1}(\Gamma)$; in fact, this is the only normal form if we equip the corresponding free shuffle reconnectad with the lexicographic ordering. Thus, according to Proposition~\ref{prop:gerst-distr}, the component $\grGerst(\Gamma)$ has a basis of form  
 \[
\mu_V^\Gamma(m_{\Gamma^*_V}\otimes b_{T_1}\otimes b_{T_2}\otimes\cdots\otimes b_{T_n}),
 \]
where $V\subset V_\Gamma$ and $T_1,T_2,\ldots,T_n$ are sets of vertices of connected components of~$\Gamma_V$. As the cochain complex $(\grGerst, \mathrm{d})$ is acyclic, each component of $\grGrav$ is spanned by the elements 
 \[
\mathrm{d}\mu_V^\Gamma(m_{\Gamma^*_V}\otimes b_{T_1}\otimes b_{T_2}\otimes\cdots\otimes b_{T_n}).
 \]
Moreover, since $\mathrm{d}$ is a reconnectad derivation, and $\mathrm{d}(b)=0$, such an element is equal to $\mu_V^\Gamma(\lambda_{\Gamma^*_V}\otimes b_{T_1}\otimes b_{T_2}\otimes\cdots\otimes b_{T_n})$.
By construction, $b=\lambda_{P_1}$, so we obtain the necessary claim. 

Next, we shall show that the claimed relations between the elements $\lambda_{\Gamma}$ actually hold. For $T\in G_\Gamma$, let us compute $\mathrm{d}(m_{\Gamma^*_T}\circ_T^{\Gamma} \lambda_{\Gamma_T})$ in two different ways. On the one hand, this is equal to $\lambda_{\Gamma^*_T}\circ_T^{\Gamma} \lambda_{\Gamma_T}$. On the other hand, this is equal to
 \[
\mathrm{d}\left(m_{\Gamma^*_T}\circ_T^{\Gamma} \left(\sum_{v\in T} m_{(\Gamma_T)^*_v}\circ_v^{\Gamma_T} b_v\right)\right)
= \mathrm{d}\left(\sum_{v\in T} m_{\Gamma^*_v}\circ_v^{\Gamma} b_v\right)=\sum_{v \in T} \lambda_{\Gamma^*_v} \circ_v^{\Gamma} \lambda_v,
 \]
showing that Relation \eqref{eq:grav1} holds. We also note that we have
 \[
0=\mathrm{d}(\lambda_{\Gamma})=\mathrm{d}\left(\sum_{v\in V_\Gamma}m_{\Gamma^*_v}\circ_v^{\Gamma} b_v\right)=\sum_{v\in V_\Gamma}\lambda_{\Gamma^*_v}\circ_v^{\Gamma} \lambda_v,
 \] 
and therefore Relation \eqref{eq:grav2} holds. Overall, this implies that there is a surjective morphism of reconnectads from the reconnectad $\calG$ with the indicated generators and relations \eqref{eq:grav1}, \eqref{eq:grav2} to the reconnectad $\grGrav$. Let us establish that it is an isomorphism. 

Let us consider the shuffle reconnectad $\calG^{\mathsf{f}}$, and equip the monomials in the corresponding free shuffle reconnectad with the lexicographic ordering. Then the leading monomial of Relation \eqref{eq:grav1} is $\lambda_{\Gamma^*_{T}}\circ_T^{\Gamma} \lambda_{\Gamma_{T}}$, and the corresponding normal monomials are
 \[
\mu_V^\Gamma(\lambda_{\Gamma^*_V}\otimes\nu_{T_1}\otimes\nu_{T_2}\otimes\cdots\otimes\nu_{T_n}),
 \]
where $V\subset V_\Gamma$, $T_1,T_2,\ldots,T_n$ are sets of vertices of connected components of $\Gamma_V$, and $\nu_{T_i}$ are some monomials obtained by iterated compositions of $\lambda_{P_1}$. The leading monomial of Relation \eqref{eq:grav2} is $\lambda_{(\Gamma)^*_{\min(V_\Gamma)}} \circ_{\min(V_\Gamma)}^{\Gamma} \lambda_{\Gamma_{\min(V_\Gamma)}}$, and this imposes further restrictions on normal monomials. First, for the subreconnectad generated by $\lambda_{P_1}$ (which is isomorphic to the desuspension reconnectad $\mathsf{S}^{-1}$), this leading term eliminates all monomials in $\lambda_{P_1}$ except for one element $\lambda^{(T)}$ per component (constructed by the same inductive rule as the element $b_T$ above). In general, normal monomials with respect to the leading terms of our relations are of the form
 \[
\mu_V^\Gamma(\lambda_{\Gamma^*_V}\otimes\lambda^{(T_1)}\otimes\lambda^{(T_2)}\otimes\cdots\otimes\lambda^{(T_n)}),
 \]
where $V\subset V_\Gamma$, $T_1,T_2,\ldots,T_n$ are sets of vertices of connected components of $\Gamma_V$, and $\min(V_\Gamma)\notin V$. By Proposition \ref{prop:normal}, these elements form a spanning set of the component $\calG(\Gamma)$. It remains to note that the number of such monomials is precisely $2^{|V_\Gamma|-1}$, which is, according to Proposition~\ref{prop:gravdim}, equal to~$\dim\grGrav(\Gamma)$, so 
 \[
\dim\calG(\Gamma)\le \dim\grGrav(\Gamma).
 \] 
Thus, the surjection $\calG\twoheadrightarrow\grGrav$ has to be an isomorphism, and moreover the defining relations of $\grGrav$ form a quadratic Gr\"obner basis of the corresponding shuffle reconnectad. By Proposition \ref{prop:shuffleKoszul} and Corollary \ref{cor:forget}, the reconnectad $\grGrav$ is Koszul.
\end{proof}

\subsection{Koszul dual of the gravity reconnectad}

In this section, we describe the Koszul dual of the gravity reconnectad. 

\begin{definition}[hypercommutative reconnectad]
The \emph{hypercommutative reconnectad} $\grHyperCom$ is the desuspension of the Koszul dual of the gravity reconnectad:
 \[
\grHyperCom := \mathsf{S}^{-1}(\grGrav)^! \ . 
 \]
\end{definition}

Let us give a presentation of the reconnectad $\grHyperCom$ by generators and relations.

\begin{proposition}\label{prop:hyperpresent}
The reconnectad $\grHyperCom$ is generated by elements $\nu_{\Gamma} \in \grHyperCom(\Gamma)$ of degree $2(|V_\Gamma|-1)$ which are invariant under $\Aut(\Gamma)$ and satisfy the following relations. For each connected graph $\Gamma$ and each $(v,v')\in E_\Gamma$, we have:
\begin{equation}\label{eq:relHyper}
\sum_{V_\Gamma\ne T\in G_\Gamma\colon v \in T} \nu_{\Gamma^*_T}\circ_T^{\Gamma} \nu_{\Gamma_T} =
\sum_{V_\Gamma\ne T\in G_\Gamma\colon v' \in T} \nu_{\Gamma^*_T}\circ_T^{\Gamma} \nu_{\Gamma_T}
\end{equation}
\end{proposition}
\begin{proof}
Let us denote by $\nu_{\Gamma}$ the generator dual to the generator $\lambda_{\Gamma}$ of the reconnectad $\grGrav$. According to the general formula for the Koszul dual reconnectad, we have 
 \[
\grHyperCom := \mathsf{S}^{-1}(\mathsf{S}^{-1}(\grGrav^{\text{!`}})^*)=\mathsf{S}^{-2}(\grGrav^{\text{!`}})^*,
 \]
so the homological degree of $\nu_{\Gamma}$ is $|\nu_{\Gamma}|=2|V_\Gamma|-2=2(|\Gamma|-1)$. Moreover, since we took the double suspension of $(\grGrav^{\text{!`}})^*$, these elements are $\Aut(\Gamma)$-invariant. One can check by a direct inspection that the relations \eqref{eq:relHyper} are ortogonal to all relations of $\grGrav$. Let us verify that we found all relations, in other words, that for each $\Gamma$ the dimension of the module of relations coincides with the dimension of weight two elements of the reconnectad $\grGrav$, given by 
$\dim H_1(\grGrav(\Gamma))=|V_\Gamma|-1$. Note that our relations say that the element 
 \[
\sum_{V_\Gamma\ne T\in G_\Gamma\colon v \in T} \nu_{\Gamma^*_T}\circ_T^{\Gamma} \nu_{\Gamma_T}  \]
does not depend on the vertex $v\in V_\Gamma$. It is clear that each of these relations is a linear combination of $|V_\Gamma|-1$ such relations corresponding to edges of any chosen spanning tree $T$ of $\Gamma$. Let us show that those $|V_\Gamma|-1$ relations are linearly independent. For that, it will be convenient to consider the corresponding shuffle reconnectad $\grHyperCom^{\mathsf{f}}$ so that for each graph $\Gamma$ we may think of its edges as directed; for each edge $e$, we denote by $s(e)$ the smaller endpoint and by $t(e)$ the larger one. This allows us to choose a concrete direction of each relation, writing 
 \[
R_e:=\sum_{s(e) \in T} \nu_{\Gamma^*_T}\circ_T^{\Gamma} \nu_{\Gamma_T} -
\sum_{t(e) \in T} \nu_{\Gamma^*_T}\circ_T^{\Gamma} \nu_{\Gamma_T}=0.
 \]
Suppose that for some choice of coefficients $c_e$ we have
 \[
\sum_{e \in E_\Gamma} c_eR_e=0.
 \] 
Collecting the coefficient of a particular monomial $\nu_{\Gamma^*_T}\circ_T^{\Gamma} \nu_{\Gamma_T}$, we see that for every vertex $v \in V_\Gamma$ the sum 
 \[
\sum_{e\colon t(e)=v} c_e - \sum_{e\colon s(e)=v} c_e
 \]
must vanish, which is a flow condition of a sort. Let us show by induction on the number of vertices that this implies that all coefficients $c_e$ of the edges of the chosen spanning tree must vanish. The case of the spanning tree of two vertices is trivial. To proceed, we note that every tree has a vertex $v_0$ which is incident to only one edge $e_0$. The flow condition at the vertex $v$ implies that we have $c_{e_0}=0$. Moreover, if we delete the vertex $v_0$ and the edge $e_0$, the remaining coefficients $c_e$ define a flow on the remaining tree, and the induction hypothesis applies. 
\end{proof}

\subsection{The wonderful Koszul pair}

We are ready to establish the main result of this section, establishing that the gravity reconnectad and the homology of the wonderful reconnectad are, up to a (de)suspension, Koszul dual to each other.

\begin{theorem}\label{th:homologyhyper}
The homology of the complex wonderful reconnectad is isomorphic to the hypercommutative reconnectad:
 \[
H_{\bullet}(\calW_{\mathbb{C}}) \cong \grHyperCom.
 \]
\end{theorem}

\begin{proof}
We know that for each variety $X(\Gamma)$ all toric orbit closures are products of smaller toric graph associahedral varieties. This immediately implies that the fundamental classes $[X(\Gamma)]$ generate the reconnectad $H_{\bullet}(\calW_{\mathbb{C}})$. 

For each graph $\Gamma$ and each edge $(v,v')\in E(\Gamma)$ there is a map 
 \[
X(\Gamma)\to X(P_1)=\mathbb{C}P^1
 \] 
given by $\{H_T\}_{T\in G_\Gamma}\mapsto H_{\{v,w\}}$. The preimages of any two points are homologous. The preimage of the point $\k^{\{v\}}$ is easily seen to be given by 
    \[
\bigcup_{V_\Gamma\ne T\in G_\Gamma\colon v\in T, v'\notin T} \overline{X(\Gamma,\{T\})} \ .  
    \]
Thus, for each edge $(v,v')\in E(\Gamma)$ we obtain the relation
 \[
\sum_{V_\Gamma\ne T\in G_\Gamma\colon v \in T, v'\notin T} [X(\Gamma^*_T)]\circ_G^{\Gamma} [X(\Gamma_T)]= 
\sum_{V_\Gamma\ne T\in G_\Gamma\colon v'\in T, v \notin T} [X(\Gamma^*_T)]\circ_G^{\Gamma} [X(\Gamma_T)],
 \]
so, adding 
 \[
\sum_{V_\Gamma\ne T\in G_\Gamma\colon v,v' \in T} [X(\Gamma^*_T)]\circ_G^{\Gamma} [X(\Gamma_T)]
 \]
to both sides of that equality, we see that there is a surjective reconnectad morphism
 \[
\grHyperCom\twoheadrightarrow H_{\bullet}(\calW_{\mathbb{C}}).
 \]
To conclude the proof, we shall argue as follows. It is well known that the Betti numbers of the smooth projective toric variety $X(\Gamma)$ are given by the coefficients of the $h$-polynomial of the dual polytope. Since that dual polytope is the graph associahedron $\calP\Gamma$, we have
 \[
\dim H_{\bullet}(X(\Gamma))=h_{\calP\Gamma}(1)=f_{\calP\Gamma}(0)=|\{\text{vertices of } \calP\Gamma\}|,
 \] 
and vertices of $\calP\Gamma$ are in one-to-one correspondence with the set $N^+_{|V_\Gamma|}(\Gamma)$ consisting of elements of $N^+(\Gamma)$ of maximal cardinality $|V_\Gamma|$. This means that to prove that the surjective map we constructed is an isomorphism, it is enough to show that for each $\Gamma$, the dimension of $\grHyperCom(\Gamma)$ does not exceed $|N^+_{|V_\Gamma|}(\Gamma)|$.

Let us consider the shuffle reconnectad $\grHyperCom^{\mathsf{f}}$ associated to the hypercommutative reconnectad. In the proof of Theorem \ref{th:gravpresent}, we established that the relations of the shuffle reconnectad $\grGrav^{\mathsf{f}}$ form a Gr\"obner basis for the lexicographic ordering of nested sets. By Proposition \ref{prop:dualPBW}, the same is true for the shuffle reconnectad $\grHyperCom^{\mathsf{f}}$ for the opposite ordering. It follows that the the shuffle reconnectad $\grHyperCom^{\mathsf{f}}$ has a basis of monomials that are
normal with respect to the minimal system of relations (found in the proof of Proposition \ref{prop:hyperpresent}) for this particular ordering, that is, the monomials 
 \[
\nu_{\Gamma_w^*}\circ_w^{\Gamma}\nu_{\Gamma_w}
 \]
for all $w\in V_\Gamma\setminus\{\min(V_\Gamma)\}$. We therefore reduced our problem to showing that for each $\Gamma$, the number of such monomials, which we shall refer to as normal monomials, and view as elements of $|N^+(\Gamma)|$, does not exceed $|N^+_{|V_\Gamma|}(\Gamma)|$.

Recall that in Section \ref{sec:rec-wond} we explained how to assign to $\tau\in N^+(\Gamma)$ a rooted tree $\mathbb{T}_\tau$ whose vertices are labelled by non-empty pairwise disjoint subsets of~$V_\Gamma$. Suppose that $\tau\in N^+_{|V_\Gamma|}(\Gamma)$. In this case each vertex is labelled by a singleton, so the set of vertices of $\mathbb{T}_\tau$ is in one-to-one  correspondence with~$V_\Gamma$. We shall denote by $V_{\tau,v}$ the unique element of $\tau$ such that the vertex of $\mathbb{T}_\tau$ corresponding to the subset $V_{\tau,v}$ is $v$. We shall call a pair $v<w$ of vertices of $\Gamma$ a \emph{descent} of~$\tau$ \cite{MR2520477} if $v$ is a child of $w$ in $\mathbb{T}_\tau$. We denote by $\Des(\tau)$ the set of descents of $\tau$. The importance of this notion for our purposes is explained by a result of \cite{MR2520477} stating that for any finite simple graph $\Gamma$, the $h$-polynomial of $\calP\Gamma$ is given by the formula
\[h_{\calP\Gamma}(t)=\sum_{\tau\in N^+_{|V_\Gamma|}(\Gamma)} t^{|\Des(\tau)|},\]
which refines the formula $\dim H_{\bullet}(\calW_{\mathbb{C}}(\Gamma))=|N^+_{|V_\Gamma|}(\Gamma)|$ and gives a useful hint on how to proceed.

We shall now construct a ``reduction'' map from the set $N^+_{|V_\Gamma|}(\Gamma)$ of ``maximal nested sets'' to normal monomials that we shall denote by $\tau\mapsto \tau^{\mathrm{red}}$ and an ``induction'' map from normal monomials to maximal nested sets that we shall denote by $\omega\mapsto \omega^{\mathrm{ind}}$.

The reduction of $\tau$ is defined by the formula
 \[
\tau^{\mathrm{red}}:=\tau\setminus\{V_{\tau,v}\mid (v,w)\in\Des(\tau)\}.
 \]
To define induction, we consider, for a normal monomial $\omega\in |N^+(\Gamma)|$, the corresponding tree $\mathbb{T}_\omega$. Let us consider all $T\in\omega$ for which $|\lambda(T)|>1$ and choose a maximal (by inclusion) element $T$ from this set. Let $v$ be the minimal (with respect to the order on $V_\Gamma$) vertex of $\lambda(T)$. There is a unique minimal (by incluson) element $T'$ containing $v$ that is compatible with the nested set $\omega$: it is the union of $\{v\}$ with all sets $S$ to which $\{v\}\cup S\in G_\Gamma$. We add $T'$ to $\omega$, and repeat the same procedure until we obtain an element of $|N^+(\Gamma)|$  for which $|\lambda(T)|=1$ for all $T$, in other words, an element of $N^+_{|V_\Gamma|}(\Gamma)$. We denote that element by $\omega^{\mathrm{ind}}$ and call it the induced maximal nested set.

\begin{lemma}\leavevmode
\begin{enumerate}
\item For any $\tau\in N^+_{|V_\Gamma|}(\Gamma)$, $\tau^{\mathrm{red}}$ is a normal monomial.   
\item For any $\omega\in N^+(\Gamma)$, we have $(\omega^{\mathrm{ind}})^{\mathrm{red}}\subset\omega$.
\item For any normal monomial $\omega\in N^+(\Gamma)$, we have $(\omega^{\mathrm{ind}})^{\mathrm{red}}=\omega$.
\end{enumerate}
\end{lemma}

\begin{proof}
Let us start with the first assertion. Suppose that $\tau^{\mathrm{red}}$ is not a normal monomial, that is, it is divisible by a leading term of our relations. We note that it is enough to consider the particular case where $\tau$ has the maximal possible number of descents, that is $|V_\Gamma|-2$: removing the elements that contain the outer element set of the leading term, or are contained in the inner element of the leading term, or are disjoint from either of them will not impact the argument. We are therefore left with the case where the result of reduction is a quadratic monomial; let us show that it is normal. There are two possibilities: $V_{\tau,\max(V_\Gamma)}=V_\Gamma$ and $V_{\tau,\max(V_\Gamma)}\ne V_\Gamma$. In the latter case, $V_{\tau,\max(V_\Gamma)}$ is never deleted in the reduction process, and hence is the only nontrivial element of $\tau^{\mathrm{red}}$, which is therefore a normal monomial. In the former case, let $T$ be the only nontrivial element of $\tau^{\mathrm{red}}$, and let $v\in\lambda(T)$. As $T$ survived the reduction process, there is a vertex $w<v$, $w\not\in T$ such that $T\cup\{w\}\in G_\Gamma$. But if our monomial were not normal, then $T$ would necessarily contain $w$, a contradiction. 

Let us now establish the second assertion. For each $v\in V_\Gamma$, we consider the corresponding element $V_{\omega^{\mathrm{ind}},v}$. We shall show that if $V_{\omega^{\mathrm{ind}},v}$ survives the reduction process, then it is an element of $\omega$. Let $w$ be the vertex for which $V_{\omega^{\mathrm{ind}},w}$ is the smallest element strictly containing $V_{\omega^{\mathrm{ind}},v}$. Note that $w\in\lambda(V_{\omega^{\mathrm{ind}},w})$. If $V_{\omega^{\mathrm{ind}},v}$ survives the reduction process, we must have $w<v$. If $V_{\omega^{\mathrm{ind}},v}$ was created during the induction process, then $v$ was at some point the minimal element of $\lambda(T)$ for some $T$. However, as $v\in V_{\omega^{\mathrm{ind}},w}$, for every $T$ such that $v\in \lambda(T)$ we also have $w\in\lambda(T)$, and so $v$ can never be minimal. Hence $V_{\omega^{\mathrm{ind}},v}$ must be an element of $\omega$ in the first place.

Let us prove the last assertion. By contraposition, we have to prove that if $(\omega^{\mathrm{ind}})^{\mathrm{red}}\neq\omega$, then the monomial $\omega$ is not normal, so that there is an element $T$ of $\omega$ that is deleted in the reduction process. Without loss of generality, no subset of $T$ belongs to $\omega$. Let $T'$ be the cardinality-minimal element of $\omega$ strictly containing $T$. Let $S$ denote  the set of elements of $T'$ that are connected to an element of $T$ by an edge of $\Gamma$. Since $T$ is deleted in the reduction process, the minimal element of $S$ is greater than the free vertex of $T$ in $\omega^{\mathrm{ind}}$. For all $u\in S$ except for the only element of $\lambda(T')$ in $\mathbb{T}_{\omega^{\mathrm{ind}}}$, the element $V_{\omega^{\mathrm{ind}},u}$ cannot survive the reduction process, and hence $u$ is greater than the only element of $\lambda(T)$ in $\mathbb{T}_{\omega^{\mathrm{ind}}}$. As no subset of $T$ belongs to $\omega$, that latter element is precisely $\max(T)$, and hence the two elements $\{T,T'\}$ give a divisor of $\omega$ that is a leading term of an element of our Gr\"obner basis. 
\end{proof}

These results together imply that for each component of the free reconnectad, the number of monomials supported at $\Gamma$ that are normal with respect to our Gr\"obner basis of $\grHyperCom^{\mathsf{f}}$ does not exceed $|N^+_{|V_\Gamma|}(\Gamma)|$. Indeed, for each such monomial $\omega$,
we have $(\omega^{\mathrm{ind}})^{\mathrm{red}}=\omega$, so the induction is injective. According to Proposition \ref{prop:normal}, this implies that
 \[
\dim \grHyperCom(\Gamma) \leq \dim H_{\bullet}(\calW_{\mathbb{C}}(\Gamma)),
 \]
so the previously constructed surjection must be an isomorphism. 
\end{proof}

\subsection{Geometrical proof of the Koszul property}

Similarly to the proof of Getzler in the case of the operad $\mathsf{HyperCom}$ \cite{MR1363058}, one can prove the Koszul property of the reconnectad $H_{\bullet}(\calW_{\mathbb{C}})$ more geometrically, illustrating the general slogan that ``operadic structures arising from compactifications with normal crossing divisors are Koszul, and the Koszul duality is the duality between the open part and the compactification''. 

\begin{proposition}\label{prop:delignekoszul}
The reconnectad $H_{\bullet}(\calW_{\mathbb{C}})$ is Koszul.
\end{proposition}

\begin{proof}
According to \cite{MR498551}, if a smooth projective complex algebraic variety $M$ is represented as $M=U\sqcup D$ where $D$ is a normal crossing divisor $D$ with components $D_1$,\ldots, $D_N$, one can define the sheaf of logarithmic differential forms $\calE_M^\bullet(\log D)$
 \[
H^\bullet(U)\cong H^\bullet(M,\calE_M^\bullet(\log D)),
 \]
and there is the Deligne spectral sequence   
 \[
E_1^{-p,q}=H^{-2p+q}(D^p,\epsilon_p)\Rightarrow E_\infty^{-p,q}=\mathrm{gr}_W H^{-p+q}(U)
 \]
where 
 \[
D^p=\bigsqcup_{i_1<\cdots<i_p} D_{i_1}\cap\cdots D_{i_p},
 \]
and $\epsilon_p$ is the locally constant line bundle which over the component $D_{i_1}\cap\cdots D_{i_p}$ is the sign representation of the group of permutations of $i_1,\ldots,i_p$ placed in homological degree $p$. 

In our case, if we take $M=X_{\mathbb{C}}(\Gamma)$ and $U=X_{\mathbb{C}}(\Gamma,\{V_\Gamma\})$, the components of $D$ are indexed by elements $V_\Gamma\ne T\in G_\Gamma$, and the intersections $D^p$ are indexed by the set $N^+_{p+1}(\Gamma)\subset N^+(\Gamma)$ consisting of elements of $N^+(\Gamma)$ of cardinality $p+1$ elements, therefore 
 \[
H^{-2p+q}(D^p,\epsilon_p)=\bigoplus_{\tau\in N^+_{p+1}(\Gamma)}H^{-2p+q}(\overline{X(\Gamma,\tau)},\epsilon_p).
 \]
The differential $d_1$ of the spectral sequence is the composition
\[\begin{tikzcd}
\bigoplus\limits_{\tau\in N^+_{p+1}(\Gamma)}H^{-2p+q}(\overline{X(\Gamma,\tau)},\epsilon_p) \arrow[swap]{d}{\mathrm{PD}} & 
\bigoplus\limits_{\tau\in N^+_{p}(\Gamma)}H^{-2p+2+q}(\overline{X(\Gamma,\tau)},\epsilon_p)  \\
\bigoplus\limits_{\tau\in N^+_{p+1}(\Gamma)}H_{2(|V_\Gamma|-1)-q}(\overline{X(\Gamma,\tau)},\epsilon_p)\arrow{r}{} 
& \bigoplus\limits_{\tau\in N^+_{p}(\Gamma)}H_{2(|V_\Gamma|-1)-q}(\overline{X(\Gamma,\tau)},\epsilon_p)\arrow{u}{\mathrm{PD}}
\end{tikzcd}
\]
where the vertical arrows are induced by the Poincar\'e duality, and the horizontal arrow is induced by inclusions of strata. Thus, we see that the first page of the Deligne spectral sequence computes the homology of the bar construction of the reconnectad $H_{\bullet}(\calW_{\mathbb{C}})$. Note that the appearance of $\epsilon_p$ corresponds to the fact that the bar construction is the cofree coreconnectad on the shift $s H_{\bullet}(\calW_{\mathbb{C}})$. 
Since the mixed Hodge structure of $H^p(X_{\mathbb{C}}(\Gamma,\{V_\Gamma\}))=H^p((\mathbb{C}^\times)^{V_\Gamma}/\mathbb{C}^\times)$ is manifestly pure of weight $2p$, we have 
 \[
E_2^{-p,q}=
\begin{cases}
H^p(X_{\mathbb{C}}(\Gamma,\{V_\Gamma\})), \text{ if } q=2p,\\
\qquad  0 \qquad\qquad \text{ otherwise}. 
\end{cases}
 \]
This implies that the reconnectad $H_{\bullet}(\calW_{\mathbb{C}})$ is Koszul, as required.
\end{proof}

As a consequence of this result, one can see that the Koszul dual reconnectad $\grGrav$ has another geometric construction, with the underlying graphical collection $\{H_{\bullet+1}(X_{\mathbb{C}}(\Gamma,\{V_\Gamma\}))\}$ and the reconnectad structure arising from the so called ``Poincar\'e residues'' \cite{MR498551}.

\bibliographystyle{plain}
\bibliography{biblio}


\end{document}